\documentclass[12pt, reqno]{amsart}
\usepackage{amsmath, amsthm, amscd, amsfonts, amssymb, graphicx, color,tikz}
\usepackage[bookmarksnumbered, colorlinks, plainpages]{hyperref}
\usepackage{mathrsfs}
\usepackage{enumerate}

\newtheorem{theorem}{Theorem}[section]
\newtheorem{lemma}[theorem]{Lemma}

\newtheorem{corollary}[theorem]{Corollary}
\theoremstyle{definition}
\newtheorem{definition}[theorem]{Definition}
\newtheorem{example}[theorem]{Example}

\newtheorem{remark}[theorem]{Remark}
\numberwithin{equation}{section}

\newcommand{\A}{\mathrm{Re\,}}
\newcommand{\B}{\mathrm{Im\,}}

\newcommand{\veps}{\varepsilon}
\newcommand{\CC}{\mathbb C}
\newcommand{\RR}{\mathbb R}
\newcommand{\DD}{\mathbb D}
\newcommand{\TT}{\mathbb T}
\newcommand{\cphiz}{C_{\varphi_0}}
\newcommand{\cphiu}{C_{\varphi_1}}
\newcommand{\puq}{\varphi_{1,\mathcal Q}}
\newcommand{\pzq}{\varphi_{0,\mathcal Q}}
\newcommand{\ttinf}{\mathbb T^\infty}
\newcommand{\tripleint}{\int_{\ttinf}\!\int_{\RR}\!\int_0^1}
\newcommand{\bpsi}{\mathcal B\psi}

\newcommand{\cha}{\mathrm{char\,}}

\allowdisplaybreaks[4]

\begin{document}
\setcounter{page}{1}

\title[Topological structure of Composition Operators]
{Topological structure of the space of composition operators on the Hardy space of Dirichlet series}
\date{\today}

\author[F. Bayart, M. Wang and X. Yao]{Fr\'ed\'eric Bayart, Maofa Wang and Xingxing Yao}

\address{Laboratoire de Math\'ematiques Blaise Pascal UMR 6620 CNRS, Universit\'e Clermont Auvergne, Campus universitaire des C\'ezeaux, 3 place Vasarely, 63178 Aubi\`ere Cedex, France.}
\email{frederic.bayart@uca.fr}

\address{School of Mathematics and Statistics, Wuhan University, Wuhan 430072, China.}\email{mfwang.math@whu.edu.cn}

\address{School of Mathematics and Physics, Wuhan Institute of Technology, Wuhan 430205, China.}\email{xxyao.math@wit.edu.cn}




\begin{abstract}
\noindent The aim of this paper is to study when two composition operators on the Hilbert space of Dirichlet series with square summable coefficients belong to the same component or when their difference is compact. As a corollary we show that if a linear combination of composition operators with polynomial symbols of degree at most 2 is compact, then each composition operator is compact.
\end{abstract}
\maketitle

\section{Introduction}

Let $H$ be a Hilbert space of analytic functions on a domain $\mathcal U$ of the complex plane
$\mathbb{C}$ and let $\varphi$ be an analytic selfmap of $\mathcal U$.
The composition operator $C_{\varphi}$ is the linear operator defined by
\begin{equation*}
  C_{\varphi}f=f\circ\varphi,~~~~ f\in H.
\end{equation*}
A central problem in the investigation of composition operators is to relate function theoretic
properties of the symbol $\varphi$ to operator theoretic properties of $C_{\varphi}$, see the books \cite{cm,sh,z}.

In recent decades, the study of the topological structure, for the operator norm topology, of the set $\mathcal C(H)$ of composition
operators acting on $H$ attracted the attention of many mathematicians.
The earliest work is Berkson's \cite{ber} isolation theorem in the setting of
Hardy-Hilbert space $H^{2}(\mathbb{\mathbb{D}})$ on the unit disk $\mathbb{D}$ in $\mathbb{C}$,
and then this result was generalized by MacCluer \cite{ma}.
The understanding of the topological structure of the set of composition operators was linked to the problem of determining
if the difference of two composition operators is compact by the Shapiro and Sundberg conjecture raised in \cite{ss}:
{$\cphiz, \cphiu$} belong to the same connected component of $\mathcal C(H^2(\mathbb D))$
if and only if $\cphiz-\cphiu$ is compact. Bourdon \cite{bourdon}
and independently Moorhouse and Toews \cite{MT} disproved this conjecture but they also showed clearly why these problems are closely related.

Beyond compact differences, the compactness of linear combinations of composition operators has also attracted  the attention of many mathematicians. As a sample of results it is known
that composition operators induced by linear fractional maps of the unit ball $\mathbb{B}_n$ in $\mathbb{C}^n$ behave quite rigidly in the following sense:
the compactness of linear combinations  $\sum_{j=1}^{N}\lambda_{j}C_{\varphi_{j}}$, for finitely many distinct linear fractional maps $\varphi_{j}$ and
nonzero complex numbers $\lambda_{j}$,
implies that each composition operator $C_{\varphi_{j}}$ is compact
on the Hardy space $H^2(\mathbb{B}_n)$ (see \cite{ckw15}).
It should be noticed that the study of the compactness of linear combinations
of composition operators is a more difficult problem than the study of the compactness of linear differences, since high order cancelations may occur, and that this often requires new ideas.
{For example, see \cite{ckw17,ckw19,km} and references therein.}

Our aim, in the present work, is to make a thorough study of these problems on the Hardy-Hilbert space $\mathcal H$  of Dirichlet series.
Hedenmalm-Lindqvist-Seip \cite{hls} introduced and started the study of the Hilbert space of Dirichlet series with square summable coefficients:
\begin{equation*}
\mathcal{H}=\bigg\{f(s)=\sum_{n=1}^{\infty}a_{n}n^{-s}:\|f\|=\bigg(\sum_{n=1}^{\infty}|a_{n}|^{2}\bigg)^{1/2}<\infty\bigg\}.
\end{equation*}
By the Cauchy-Schwarz inequality, the functions in $\mathcal{H}$ are all
analytic on the half-plane $\mathbb{C}_{1/2}$ (where, for $\theta\in\mathbb{R}$, $\mathbb{C}_{\theta}=\{s\in\mathbb{C}: \A s>\theta\}$ and $\mathbb C_+=\mathbb C_0$).
A characteristic feature of the space $\mathcal{H}$ is its reproducing kernel function being essentially the Riemann zeta function:
$K_{w}(s)=\zeta(\overline{w}+s)$ for $w,s\in \mathbb{C}_{1/2}$, where $\zeta(z)=\sum_{n=1}^{\infty}n^{-z}$.
The main results in \cite{gh,qs} show that
an analytic map $\varphi:\mathbb{C}_{1/2}\rightarrow\mathbb{C}_{1/2}$ induces a bounded composition operator $C_{\varphi}$ on $\mathcal{H}$ if and only if it is a member of the Gordon-Hedenmalm class $\mathcal{G}$ defined as follows. \medskip

\noindent \textbf{Definition.}
The \textbf{Gordon-Hedenmalm class} $\mathcal{G}$ consists of the maps $\varphi:\mathbb{C}_{1/2}\rightarrow\mathbb{C}_{1/2}$ of the form
$$\varphi(s)=c_{0}s+\psi(s),$$
where $c_{0}$ is a non-negative integer (called the \textbf{characteristic} of $\varphi$, i.e., $\cha(\varphi)=c_{0}$),
and $\psi(s)=\sum_{n=1}^{\infty}c_{n}n^{-s}$ converges uniformly in $\mathbb{C}_{\varepsilon}$ for every $\varepsilon>0$ and has the following properties:
\begin{itemize}
  \item[(a)] If $c_{0}=0,$ then $\psi(\mathbb{C}_{0})\subseteq\mathbb{C}_{1/2}.$
  \item[(b)] If $c_{0}\geq1,$ then either $\psi\equiv c_1$ with $\A c_1\geq 0$ or $\psi(\mathbb{C}_{0})\subseteq\mathbb{C}_{0}.$
\end{itemize}
\medskip

From then on, several authors have studied the properties of composition operators acting on $\mathcal{H}$
or on similar spaces of Dirichlet series, see for instance \cite{bai,baib,b1,b2,bb,bp,qs}. 
In particular, the study of the compactness of composition operators on $\mathcal H$ is a difficult problem. Very recently, when the symbol has zero characteristic, Brevig and Perfekt \cite{bp20} gave
a characterization of compactness by introducing a mean-counting function which has to satisfy
a decay condition. Nevertheless, a similar characterization is unknown for symbols with positive characteristic, and even for simple symbols with zero characteristic, subtle phenomena occur.

Let us introduce some terminology.
A symbol $\varphi\in\mathcal G$ is said to have \textbf{unrestricted range} if
\begin{equation*}
  \inf_{s\in \mathbb{C}_{0}}\A \varphi(s)=\begin{cases}1/2  & \text{if}\ c_{0}=0, \cr 0 &\text{if}\ c_{0}\geq1.\end{cases}
\end{equation*}
Correspondingly, if $\varphi(\mathbb{C}_{0})$ is strictly contained in any smaller half-plane, we say that $\varphi$ has \textbf{restricted range}.
It is well-known that the composition operator $C_{\varphi}$ is compact on $\mathcal{H}$  when $\varphi$ has restricted range \cite[Theorem 21]{b1}.
We call $\varphi$ a \textbf{Dirichlet polynomial symbol}, if
\begin{equation*}\label{dps}
  \varphi(s)=c_{0}s+c_{1}+\sum_{k=2}^{n}c_{k}k^{-s}.
\end{equation*}
In this work, we will be interested in polynomial symbols satisfying some conditions on their degree. Recall the definition of the arithmetic function $\Omega(n):=\alpha_1+\cdots+\alpha_d$ provided $n=p_1^{\alpha_1}\cdots p_d^{\alpha_d}$ where $\mathcal P=(p_i)_{i\geq 1}$ denotes the set of prime numbers. If
$\psi(s)=\sum_{k=1}^n c_n k^{-s}$ is a Dirichlet polynomial, its {\bf degree} is the infimum of the integers $N$ such that $c_ n=0$ provided $\Omega(n)>N$.

We will also be interested in composition operators $C_{\varphi}$ generated by \emph{linear symbols}
\begin{equation*}
\varphi(s)=c_{0}s+c_{1}+\sum_{k=1}^{d}c_{q_{k}}q_{k}^{-s},
\end{equation*}
where $\{q_{k}\}_{k=1}^{d}$ is a sequence of finitely multiplicatively independent positive integers,
and $q_{k}\neq1$, $c_{q_{k}}\neq0$ for each $k=1,\cdots,d$. It should be observed that a linear symbol is not necessary
a polynomial symbol of degree less than or equal to $1$.
Following \cite{b2} and \cite{fi},
it is known that the composition operator $C_\varphi$ induced by the linear symbol $\varphi$ with $c_{0}=0$ is compact
if and only if $\varphi$ has restricted range or $d\geq2$;
when $c_{0}\geq1$, $C_{\varphi}$ is compact if and only if $\varphi$ has restricted range.
For a Dirichlet polynomial symbol $\varphi$ with $c_{0}\geq1$,
Bayart-Brevig \cite{bb} showed that $C_{\varphi}$ is compact if and only if $\varphi$ has restricted range.

Let us come back to our initial problems of the topological structure of $\mathcal C (\mathcal H)$
and of the compactness of {linear combinations of composition operators} on $\mathcal H$.
The first result in this direction was obtained by the authors in \cite{bwy}:
if $\varphi_1$ and $\varphi_2$ are linear symbols and $\lambda_1,\ \lambda_2$ are nonzero complex numbers,
then $\lambda_1\cphiu+\lambda_2 C_{\varphi_2}$ is compact if and only if both $\cphiu$ and $C_{\varphi_2}$ are compact.
We extend this result in two directions, allowing any finite linear symbols,
and allowing polynomial symbols of degree at most 2 (see Theorems \ref{thm:lcpolynomial}, \ref{thm:lclinear} and \ref{thm:lclinearzero}).

To prove this, we will develop a machinery, inspired by \cite{MT},
allowing to get lower bounds for the essential norm of a linear combination of composition operators.
Many problems arise in the setting of $\mathcal H$: the lack of reproducing kernels defined on $\mathbb C_+$,
the infinite polydisc $\DD^\infty$ hidden behind $\CC_+$,
the characteristic of the symbol which forces to adopt different proofs following its value.
For instance, we will show that a basic result of Shapiro and Sundberg (the set of compact composition operators on $H^2(\DD)$ is connected) does not extend to $\mathcal C(\mathcal H)$,
but that the set of compact composition operators with prescribed characteristic is connected.

Most of these results are necessary conditions for the difference of two composition operators to be compact (or for two operators to be in the same connected component). In Section \ref{sec:sufficient} we turn to sufficient conditions. We succeed (see Theorem \ref{thm:samecomponent} for instance)
to give sufficient conditions that are easy to testify to prove that $\cphiz-\cphiu$ is compact, or that
$\cphiz$ and $\cphiu$ belong to the same connected component.
These results involve sufficient conditions of independent interest to prove that some weighted composition operators are bounded or compact. Again the cases of zero and nonzero characteristics
require quite different approaches, one based on Carleson measures, one based on Nevanlinna counting functions. Finally we mention that in some cases, the necessary and sufficient conditions coincide, giving a characterization of the compactness of $\cphiz-\cphiu$.

\medskip

Throughout this paper we use the notation $X\lesssim Y$  or $Y\gtrsim X$ for non-negative functions $X$ and
$Y$ to mean that there exists $C>0$ such that $X\le CY$, where $C$ does not depend on the associated variables.
Similarly, we use the notation $X\approx Y$ if both $X\lesssim Y$ and $Y \lesssim X$ hold.
We will sometimes write $X\lesssim_a Y$ if we want to point out that the involved constant depend on $a$.
Given two positive functions $f(x)$ and $g(x)$ defined on $(a,+\infty)$,
we write $f(x)\sim_{x\to a}g(x)$ (respectively, $f(x)\sim_{x\to +\infty}g(x)$),
provided $\lim_{x\to a}\frac{f(x)}{g(x)}=1$ (respectively, $\lim_{x\to +\infty}\frac{f(x)}{g(x)}=1$).

\medskip
\section{Preliminaries}

This section first introduces {some background material} associated to Bohr lifts,
vertical limit functions,
reproducing and partial reproducing kernels.

\subsection{Bohr lifts}
At some places, it will be useful to work on the polydisc $\mathbb D^d$ instead of $\mathbb C_0$.
Let $f(s)=\sum_{n=1}^{\infty}a_n n^{-s}$ be a Dirichlet series
and let $\mathcal Q=\{q_1,\dots,q_d\}$ be a sequence of multiplicatively independent integers such that
for each $n$ with $a_n\neq 0$,
we have $n=q^\alpha:=q_1^{\alpha_1}\cdots q_d^{\alpha_d}$ with $\alpha\in\mathbb N_0^d$.
We will say that $f$ has $d$ multiplicative generators when $d<+\infty$, and that $\mathcal Q$ is compatible with $f$.
The associated Bohr lift $\mathcal B_{\mathcal Q}f(z)$ is the holomorphic function defined on $\mathbb{D}^d$ by
$$\mathcal B_\mathcal Q f(z)=\sum_{n=1}^{\infty} a_n z_1^{\alpha_1}\cdots z_d^{\alpha_d}.$$
By Kronecker's theorem (see for example \cite{qq}),
the map $\mathbb R\to\mathbb T^{d}$, $t\mapsto (q_1^{it},\dots,q_d^{it})$ has dense range. In particular,
if $f$ maps $\mathbb C_0$ into $\mathbb C_0$, then $\mathcal B_\mathcal Q f$ maps $\mathbb{D}^{d}$ into $\mathbb C_0$,
and $\overline{f(i\mathbb R)}=\overline{\mathcal B_\mathcal Q f(\mathbb T^d)}$.

For any Dirichlet series $f$ having finitely multiplicative generators, for all $d$ sufficiently large,
the sequence of the first $d$ prime numbers $\{p_1,\dots,p_d\}$ is always compatible with $f$. This Bohr lift will be simply denoted by $\mathcal B$. It induces a unitary map from $\mathcal H$ onto the Hardy space $H^2(\mathbb{D}^\infty)$ over the infinite polydisc $\mathbb{D}^\infty$.
It should be observed that the degree of a Dirichlet polynomial $f$ is exactly the degree of its Bohr lift $\mathcal Bf$.

\subsection{Vertical limit functions}
Following \cite{hls}, let $\Xi$ be the dual group of $\mathbb{Q}_{+}$,
where $\mathbb{Q}_{+}$ denotes the multiplicative discrete group of strictly positive rational numbers.
Consequently, $\Xi$ is the set of all functions $\chi:\mathbb{Q}_{+}\rightarrow \mathbb{C}$, such that
\begin{itemize}
  \item[(a)] $\chi(mn)=\chi(m)\chi(n)$ for all $m,n$ in $\mathbb{Q}_{+}$.
  \item[(b)] $|\chi(n)|=1$.
\end{itemize}
$\Xi$ can be identified with $\mathbb{T}^{\infty}$, the Cartesian product of countably many copies of the unit circle.
Indeed, given any point $z=(z_{1},z_{2},\ldots)\in \mathbb{T}^{\infty}$,
we define the value of $\chi$ at the primes through
\begin{equation*}
  \chi(2)=z_{1},\chi(3)=z_{2},\ldots,\chi(p_{m})=z_{m},\ldots
\end{equation*}
and extend the definition multiplicatively.
This then yields a character, and clearly all characters can be obtained in this procedure.
Also the unique normalized Haar measure on $\Xi$ can be identified with the ordinary product measure $m$ on $\mathbb{T}^{\infty}$.

For $f(s)=\sum_{n=1}^{\infty} a_{n}n^{-s}\in\mathcal{H}$ and $\chi\in \mathbb{T}^{\infty}$,
the vertical limit function $f_{\chi}$ is defined by
\begin{equation*}
  f_{\chi}(s):=\sum_{n=1}^{\infty} a_{n}\chi(n)n^{-s}.
\end{equation*}
Note that the vertical translate function $T_{\tau}f(s):=f(s+i\tau)$, $\tau\in \mathbb{R}$,
corresponds to $f_\chi$ with $\chi(n)=n^{-i\tau}$. The name vertical limit function is justified by \cite[Lemma 2.4]{hls},
which asserts that the functions $f_{\chi}$ are precisely those obtained from the Dirichlet series $f$ by taking a limit of vertical translations,
\begin{equation*}
  f_{\chi}(s)=\lim_{k\rightarrow\infty}T_{\tau_{k}}f(s).
\end{equation*}
Here the convergence is uniform on compact subsets of the half-plane where $f$ converges uniformly.
The vertical limit function $f_{\chi}$ sometimes has better properties than the original function $f$.
As explained in \cite[Section 4.2]{hls}, if $f$ is in $\mathcal{H}$,
the Dirichlet series $f_{\chi}$ converges in $\mathbb{C}_{0}$ for almost every $\chi\in \mathbb{T}^{\infty}$,
and the non-tangential boundary value
\begin{equation*}
  f^{\ast}_\chi(it):=\lim_{\sigma\rightarrow0^{+}}f_{\chi}(\sigma+it)
\end{equation*}
exists for almost every $\chi\in \mathbb{T}^{\infty}$ and almost all $t\in\mathbb R$.
By \cite[Theorem 4.1]{hls}, we can compute the norm of $f$ in terms of the function $f_{\chi}$:
if $\mu$ is a finite Borel measure on $\mathbb{R}$, then
\begin{equation}\label{eq:normH0}
  \|f\|^{2}\mu(\mathbb{R})=\int_{\mathbb{T}^{\infty}}\int_{\mathbb{R}}|f_{\chi}(it)|^{2}d\mu(t)dm(\chi).
\end{equation}

For any symbol $\varphi(s)=c_{0}s+\psi(s)$,
define $\varphi_{\chi}$ as $\varphi_{\chi}(s)=c_{0}s+\psi_{\chi}(s)$ for $\chi\in \mathbb{T}^{\infty}$,
a vertical limit of the functions $\varphi_{\tau}(s)=c_{0}s+\psi(s+i\tau)$.
As is clarified in \cite{gh}, for any such symbol $\varphi:\mathbb{C}_{1/2}\rightarrow\mathbb{C}_{1/2}$,
and $f\in\mathcal{H}$, $\chi\in \mathbb{T}^{\infty}$, the following relation holds:
\begin{equation*}
  (f\circ\varphi)_{\chi}(s)=f_{\chi^{c_{0}}} \circ \varphi_{\chi}(s),\ \ s\in \mathbb{C}_{1/2}.
\end{equation*}
Especially, for any Dirichlet series $\varphi\in\mathcal{G}$ and $f\in\mathcal{H}$,
$C_{\varphi}f$ converges uniformly in $\mathbb{C}_{\varepsilon}$ for every $\varepsilon>0$,
and $(C_{\varphi}f)_{\chi}(s)=C_{\varphi_{\chi}}f_{\chi^{c_0}}(s)$.
This implies that $\|C_{\varphi_{\chi}}f\|=\|C_{\varphi}f\|$, and thus $\varphi_{\chi}$ belongs to $\mathcal{G}$ for every $\chi\in \mathbb{T}^{\infty}$.

\subsection{Reproducing and partial reproducing kernels}\label{sec:rk}

As recalled in the introduction, the reproducing kernel at $w\in\mathbb C_{1/2}$ in $\mathcal H$ is defined by $K_w(s)=\zeta(\overline w+s)$. Reproducing kernels are often useful in the study of composition operators (especially in the study of their compactness) since $C_\varphi^*(K_w)=K_{\varphi(w)}$. Moreover, we know that
$$\|K_w\|^2 =\zeta(2 \A w)\sim_{\A w\to 1/2}\frac{1}{2\A w -1}$$
and that if $\{w_k\}$ is any sequence of $\mathbb C_{1/2}$ such that $\A w_k\to 1/2$,
then $K_{w_k}/\| K_{w_k}\|$ tends weakly to $0$.
%
%

It turns out that in the specific case of composition operators acting on $\mathcal H$,
the reproducing kernels $K_w$ have a serious drawback:
they are limited to $\mathbb C_{1/2}$ whereas the symbol $\varphi$ may be extended to $\mathbb C_0$.
This motivates the introduction of partial reproducing kernels (refer to \cite{b2} for more details).
For $w\in \mathbb{C}_0$, and $d\geq1$, the partial reproducing kernel of order $d$ at $w$ is defined by
\begin{equation*}
  K_{d,w}(s)=\prod_{j=1}^{d}\left(\sum_{n\geq1}p_{j}^{-n(\overline{w}+s)}\right)=
  \prod_{j=1}^{d}\frac{1}{1-p_{j}^{-(\overline{w}+s)}}
  =\sum_{n\geq1,\ p^{+}(n)\leq p_{d}}n^{-\overline{w}}n^{-s}
\end{equation*}
where $p^+(n)$ denotes the biggest prime factor of $n$.
The function $K_{d,w}$ reproduces partially $\mathcal{H}$:
if $f(s)=\sum_{n=1}^{\infty}a_{n}n^{-s}\in \mathcal{H}$, then
\begin{equation*}
  \langle f, K_{d,w}\rangle=\sum_{n\geq1,\ p^{+}(n)\leq p_{d}}a_{n}n^{-w}.
\end{equation*}
By Euler's identity, we have
\begin{equation*}
  \|K_{d,w}\|=\prod_{j=1}^{d}\left(\frac{1}{1-p_{j}^{-2\A w}} \right)^{1/2}=:\zeta_{d}(2\A w)^{1/2}.
\end{equation*}
Moreover, it is clear that $k_{d,w_k}:=K_{d,w_k}/\|K_{d,w_k}\|$ converges weakly to zero for any sequence $\{w_k\}\subseteq\mathbb C_0$ satisfying  $\A w_k\rightarrow0$ and that
$$\|K_{d,w}\|^2\sim_{\A w\to 0} \frac{A_{d}}{(\A w)^{d}}$$
for some positive constant $A_d$ depending only on $d$.

The interest of using partial reproducing kernels in the study of composition operators comes from \cite[Proposition 5]{b2}: if $\varphi(s)=c_{0}s+\sum_{n=1}^{\infty}c_{n}n^{-s}$ belongs to $\mathcal{G}$ with $c_{n}=0$ for $p^{+}(n)>p_{d}$,
then $C^{\ast}_{\varphi}K_{d,w}=K_{d, \varphi(w)}$ when $c_{0}\geq1$,
and $C^{\ast}_{\varphi}K_{d,w}=K_{\varphi(w)}$ when $c_{0}=0$.
We will see later (Lemma \ref{lem:prk}) how to generalize this formula to any symbol.

We shall also use a variant of these partial reproducing kernels. For an integer $q\geq 2$ and $w\in\mathbb C_0$,
we set $K_w^{(q)}(s)=(1-q^{-\overline w}q^{-s})^{-1}$, $s\in\mathbb C_{1/2}$. Clearly, $\|K_w^{(q)}\|\sim C_q (\A w)^{-1/2}$
as $\A w\to 0$ and, setting $k_w^{(q)}=K_w^{(q)}/\|K_w^{(q)}\|$, any sequence $(k_{w_k}^{(q)})$ converges weakly to zero
in $\mathcal H$ as soon as $\A w_k$ goes to zero. The following lemma comes from \cite{bwy}.

\begin{lemma}\label{lem:otherk}
 Let $r\geq 2$ be an integer and $\varphi(s)=c_1+c_r r^{-s}\in\mathcal G$.
 \begin{enumerate}[(i)]
  \item Assume that $r=q^n$ for some positive integer $n$. Then $C_\varphi^*(K_w^{(q)})=K_{\varphi(w)}$.
  \item Assume that $r^m\neq q^n$ for all positive integers $n$ and $m$. Then $C_\varphi^*(K_w^{(q)})=K_{c_1}$.
 \end{enumerate}
\end{lemma}

\section{Connected components in the space of composition operators}

\subsection{The role of the characteristic}

It is a result of Shapiro and Sundberg \cite{ss} that the collection of compact composition operators on $H^2(\DD)$ is arcwise connected. This result breaks down for composition operators on $\mathcal H$ because of the characteristic: two composition operators with different characteristics cannot belong to the same component.

\begin{theorem}\label{thm:differentchar}
The map $\mathcal C(\mathcal H)\to\mathbb N$, $C_\varphi\mapsto \cha(\varphi)$ is continuous.
\end{theorem}
We prove this result by a reasoning on the coefficients of $p^{-\varphi(s)}$ which can be adapted to a much more general context than $\mathcal H$. We need to introduce the following definition.

\begin{definition}
Let $l\geq 1$. A {\bf weighted partition} of $l$ is a triple $(r,i,\gamma)$ with $r\geq 1$, $i\in\{1,\dots,l\}^r$ satisfying $i_1<i_2<\cdots<i_r$, $\gamma\in \mathbb N^r$ and $\gamma_1 i_1+\cdots+\gamma_r i_r=l$.
The set of weighted partitions of $l$ will be denoted by $\mathcal{WP}_l$.
\end{definition}

\begin{lemma}\label{lem:differentchar}
Let $l\geq 1$ and let $p$ be a prime number. For each $(r,i,\gamma)$ a weighted partition of $l$,
there exists a real number $u(p,l,r,i,\gamma)$ such that, for all sequences $c=(c_k)_{k\geq 2}$ with $\sigma_a\left(\sum_k c_k k^{-s}\right)<+\infty$, for all prime numbers $q$,
$$p^{-\sum_{k\geq 2}c_k k^{-s}}=\sum_{k=1}^{+\infty} \alpha_k(c) k^{-s}$$
with
$$\alpha_{q^l}(c)=-c_{q^l}\log p+\sum_{\substack {(r,i,\gamma)\in\mathcal{WP}_l\\ i_r<l}} u(p,l,r,i,\gamma)
(c_{q^{i_1}})^{\gamma_1}\cdots (c_{q^{i_r}})^{\gamma_r}.$$
\end{lemma}
\begin{proof}
We can write formally
$$p^{-\sum_{k\geq 2}c_k k^{-s}}=\prod_{k\geq 2}\left(\sum_{n=0}^{+\infty} \frac{(-1)^n (c_k)^n}{n!}
(\log p)^n (k^n)^{-s}\right).$$
If we expand the infinite product and since $q$ is prime, the term $(q^l)^{-s}$ can only be obtained as the sum of products of the terms $\big( (q^{i_1})^{\gamma_1}\cdots  (q^{i_r})^{\gamma_r} \big)^{-s}$
with $\gamma_1 i_1+\cdots+\gamma_r i_r=l$. More precisely,
$$u(p,l,r,i,\gamma)=\frac{(-1)^{\gamma_1+\cdots+\gamma_r}}{\gamma_1! \cdots \gamma_r!}(\log p)^{\gamma_1+\cdots+\gamma_r}.$$
That this formal computation leads to a convergent Dirichlet series follows from the absolute convergence of the involved Dirichlet series in some half-plane.
\end{proof}

\begin{proof}[Proof of Theorem \ref{thm:differentchar}]
Let $\varphi(s)=c_0s+c_1+\sum_{k\geq 2}c_k k^{-s}\in\mathcal C(\mathcal H)$ and let $(C_{\varphi_j})\in\mathcal C(\mathcal H)$ converging to $C_\varphi$ in the operator norm topology. In order to proceed by contradiction, we assume that $(\cha(\varphi_j))$ does not converge to $c_0$. Extracting if necessary, we may assume that $\cha(\varphi_j)\neq c_0$ for all $j$. We first handle the easiest case $c_0=0$.
Since, for all $j$,
\begin{align*}
\|2^{-\varphi}-2^{-\varphi_j}\|&\geq \lim_{\sigma\to+\infty} |2^{-\varphi(\sigma)}-2^{-\varphi_j(\sigma)}|\\
&\geq 2^{-c_1}>0,
\end{align*}
we get that $(C_{\varphi_j})$ cannot converge to $C_\varphi$. Hence, we assume $c_0\geq 1$.

We write $\varphi_j(s)=d_0(j)s+d_1(j)+\sum_{k\geq 2}d_k(j) k^{-s}$.
 Let $p,q$ be two prime numbers and $l\geq 0$.  We write
$$p^{-\varphi(s)}=\sum_{k\geq 1}a_k k^{-s}\textrm{ and }p^{-\varphi_j(s)}=\sum_{k\geq 1}b_k(j) k^{-s}.$$
Let us first assume that $\cha(\varphi_j)> c_0$ for an infinite number of integers $j$, hence for all $j$ by extraction.
Then $a_{p^{c_0}}=p^{-c_1}$ whereas $b_{p^{c_0}}(j)=0$. This contradicts that $p^{-\varphi_j}$
tends to $p^{-\varphi}$ in $\mathcal H$.

Therefore it remains to handle the case $\cha(\varphi_j)<c_0$ for all $j$. Upon another extraction we may and
shall assume that $d_0(j)=d_0\in\{0,\dots,c_0-1\}$.
When $p\neq q$, for all $l\geq 1$, $a_{p^{d_0}q^l}=0$ whereas $b_{p^{d_0}q^l}(j)=p^{-d_1(j)}\alpha_{q^l}(d(j))$.
Hence
\begin{equation}\label{eq:dc1}
p\neq q,\ l\geq 1\implies p^{-d_1(j)}\alpha_{q^l}(d(j))\xrightarrow{j\to+\infty}0.
\end{equation}
When $p=q$, setting $m=c_0-d_0$, then $a_{q^{c_0}}=q^{-c_1}$
whereas $b_{q^{c_0}}(j)=q^{-d_1(j)}\alpha_{q^m}(d(j))$.  Thus
\begin{equation}\label{eq:dc2}
q^{-d_1(j)}\alpha_{q^m}(d(j))\xrightarrow{j\to+\infty}q^{-c_1}.
\end{equation}
We then fix an increasing sequence of prime numbers $(q_l)_{1\leq l\leq m+1}$ such that, for all $l=1,\dots,m$, $q_{l+1}>(q_l)^{l+1}$. We show by a finite induction that, for all $l=1,\dots,m$,
\begin{equation}\label{eq:dc3}
q_l^{-d_1(j)} d_{q_{m+1}^l}(j)\xrightarrow{j\to+\infty}0.
\end{equation}
For $l=1$, this is \eqref{eq:dc1} applied with $p=q_1$, $q=q_{m+1}$ and we use Lemma \ref{lem:differentchar} to compute the value of $\alpha_{q_{m+1}}(d(j))$. Let us assume that \eqref{eq:dc3} has been obtained for $l-1$
and let us prove it for $l$. We apply again \eqref{eq:dc1} with $p=q_l$ and $q=q_{m+1}$ so that
$$
\begin{array}{l}
q_l^{-d_1(j)}\Bigg(-d_ {q_{m+1}^l}(j)\log(q_l)+\\
\displaystyle\sum_{\substack{(r,i,\gamma)\in\mathcal{WP}_l \\ i_r<l}}u(q_l,l,r,i,\gamma)
(d_{q_{m+1}^{i_1}}(j))^{\gamma_1}\cdots (d_{q_{m+1}^{i_r}}(j))^{\gamma_r}\Bigg)\xrightarrow{j\to+\infty}0.
\end{array}
$$
Now since $q_l>q_{i_1}^{\gamma_1}\cdots q_{i_r}^{\gamma_r}$ for all $(r,i,\gamma)\in\mathcal{WP}_l$ with $i_r<l$, we may write
$$q_l^{-d_1(j)}|d_{q_{m+1}^{i_1}}(j)|^{\gamma_1}\cdots |d_{q_{m+1}^{i_r}}(j)|^{\gamma_r}\leq
|q_{i_1}^{-d_1(j)}d_{q_{m+1}^{i_1}}(j)|^{\gamma_1}\cdots |q_{i_r}^{-d_1(j)}d_{q_{m+1}^{i_r}}(j)|^{\gamma_r}$$
and using the induction hypothesis we obtain that
$$q_l^{-d_1(j)}(d_{q_{m+1}^{i_1}}(j))^{\gamma_1}\cdots (d_{q_{m+1}^{i_r}}(j))^{\gamma_r}\xrightarrow{j\to+\infty}0.$$
This yields \eqref{eq:dc3} at rank $l$.

Let us now conclude. Property \eqref{eq:dc2} for $q=q_{m+1}$ tells us that
$$
\begin{array}{l}
q_{m+1}^{-d_1(j)}\Bigg(-d_ {q_{m+1}^m}(j)\log(q_{m+1})+\\
\displaystyle\sum_{\substack{(r,i,\gamma)\in\mathcal{WP}_m\\
i_r<m}}u(q_{m+1},m,r,i,\gamma)
(d_{q_{m+1}^{i_1}}(j))^{\gamma_1}\cdots (d_{q_{m+1}^{i_r}}(j))^{\gamma_r}\Bigg)\xrightarrow{j\to+\infty}q_{m+1}^{-c_1}.
\end{array}
$$
But since $q_{m+1}\geq q_m$, \eqref{eq:dc3} implies that $q_{m+1}^{-d_1(j)}d_ {q_{m+1}^m}(j)$ tends to zero. Furthermore we can repeat the argument of the inductive step to get
$$q_{m+1}^{-d_1(j)}\sum_{\substack{(r,i,\gamma)\in\mathcal{WP}_m\\
i_r<m}}u(q_{m+1},m,r,i,\gamma)
(d_{q_{m+1}^{i_1}}(j))^{\gamma_1}\cdots (d_{q_{m+1}^{i_r}}(j))^{\gamma_r}\xrightarrow{j\to+\infty}0.$$
This yields the contradiction we are looking for.
\end{proof}

\begin{remark}
In fact our proof shows that the map $\mathcal C(\mathcal H)\to\mathbb N$,
{$C_\varphi\mapsto \textrm{char}(\varphi)$ is continuous }
even if we endow $\mathcal C(\mathcal H)$ with the weak operator topology.
\end{remark}

\subsection{Compact composition operators}

We now show that the set of compact composition operators with prescribed characteristic is connected. The proof is necessary different from the proof on $H^2(\mathbb D)$ since we cannot connect any compact composition operator to a composition operator with constant symbol.

\begin{theorem}\label{thm:connectedcompact}
Let $c_0\in\mathbb N_0$. The set of compact composition operators with characteristic equal to $c_0$ is arcwise connected.
\end{theorem}
\begin{proof}
We fix $\varphi\in\mathcal G$ with $\cha(\varphi)=c_0$ such that $C_\varphi$ is compact. We first show that there exists an arc of compact composition operators between $C_\varphi$ and $C_{\tilde\varphi}$ for some $\tilde \varphi\in\mathcal G$ with $\tilde\varphi(\CC_+)\subseteq\CC_{\frac 12+\veps}$ for some $\veps>0$, $\cha(\tilde\varphi)=c_0$ and,
writing $\tilde\varphi=c_0s+\tilde\psi$, $\tilde\psi(\CC_+)$ is bounded. Define $\varphi_\sigma(\cdot)=\varphi(\cdot+\sigma)$. We claim
that $\sigma\mapsto C_{\varphi_\sigma}$ is continuous (for the operator norm topology).
The choice $\tilde\varphi=\varphi_1$ will answer the problem: for $c_0=0$, $\varphi_1$ has restricted range,
so that $\varphi_1(\CC_+)\subseteq \CC_{\frac 12+\veps}$; for $c_0\geq 1$, we even have $\varphi(\CC_+)\subseteq \CC_1$.
Moreover, if $\varphi=c_0s+\psi$, since the abscissa of uniform convergence of $\psi$ is nonpositive,
its abscissa of absolute convergence does not exceed $1/2$.
Hence, $\psi(\cdot+1)$ is an absolutely convergent Dirichlet series in $\overline{\CC_+}$ which implies that $\tilde\psi(\CC_+)$ is bounded. Finally, since for any $\sigma>0$,
$\varphi_\sigma$ has restricted range, $C_{\varphi_\sigma}$ is compact.

Let $T_\sigma$ be the composition operator induced by the map $s\mapsto s+\sigma$, $\sigma\geq 0$, so that $C_{\varphi_\sigma}=T_\sigma\circ C_\varphi$. Now observe that for a fixed $f\in\mathcal H$ and a fixed $\sigma_0\geq 0$, $T_\sigma f\to T_{\sigma_0}f$ provided $\sigma\to\sigma_0$.
Furthermore, $\{C_\varphi(f):\ \|f\|\leq 1\}$ is  relatively compact in $\mathcal H$ and the family $(T_\sigma)_{\sigma\geq 0}$ is equicontinuous since $\|T_\sigma\|\leq 1$ for all $\sigma\geq 0$.
It follows from a standard compactness argument that the pointwise convergence of  $T_\sigma f$ to $ T_{\sigma_0}f$ when $\sigma$ tend to $\sigma_0$ is uniform on $\{C_\varphi(f):\ \|f\|\leq 1\}$.
Namely for all $\veps>0$ there exists $\delta>0$ such that
$$\|f\|\leq 1 \textrm{ and }|\sigma-\sigma_0|<\delta \implies \|C_{\varphi_\sigma}(f)-C_{\varphi_{\sigma_0}}(f)\|\leq\veps.$$
This proves the continuity of $\sigma\mapsto C_{\varphi_\sigma}$.

To finish the proof, we only need to prove that, given $\veps>0$ and $C>0$, for two symbols $\varphi_0=c_0s+\psi_0$,
$\varphi_1=c_0s+\psi_1$ such that $\varphi_0(\CC_+),\varphi_1(\CC_+)$ are contained in ${\CC_{\frac 12+\veps}}$ and {$|\psi_{0}(s)|,|\psi_1(s)|\leq C$ for all $s\in\CC_+$,}
there exists an arc in $\mathcal C(\mathcal H)$ between $\cphiz$ and $\cphiu$ consisting in compact composition operators.
Define, for $\lambda\in(0,1)$, $\varphi_\lambda=c_0s +(1-\lambda)\psi_0+\lambda\psi_1$.
Again it suffices to show that $\lambda\mapsto C_{\varphi_\lambda}$ is continuous since each $\varphi_\lambda$ has restricted range. Pick $f\in\mathcal H$ with $\|f\|\leq 1$ and write
\begin{align*}
\left\|C_{\varphi_\lambda}(f)-C_{\varphi_{\lambda'}}(f)\right\|^2&=\int_{\ttinf}|(f\circ \varphi_\lambda)_\chi(0)-
(f\circ \varphi_{\lambda'})_\chi(0)|^2dm(\chi)\\
&=\int_{\ttinf}|f_{\chi^{c_0}}((\varphi_\lambda)_\chi(0))-f_{\chi^{c_0}}((\varphi_{\lambda'})_\chi(0))|^2dm(\chi).
\end{align*}
Since $(\varphi_\lambda)_\chi$ is a vertical limit of $\varphi_\lambda$, for all $\chi\in\ttinf$, $(\varphi_\lambda)_\chi(0)$ and $(\varphi_{\lambda'})_\chi(0)$ belong to $\overline{\CC_{\frac 12+\veps}}$. Since $\|f_{\chi^{c_0}}\|\leq 1$,
there exists some $M>0$, depending only on $\veps$, such that, for all $\chi\in\ttinf$,
$$\sup_{s\in \overline{\CC_{\frac 12+\veps}}} |f_{\chi^{c_0}}'(s)|\leq M.$$
Therefore
\begin{align*}
\left\|C_{\varphi_\lambda}(f)-C_{\varphi_{\lambda'}}(f)\right\|^2&\leq M^2 \int_{\ttinf} |(\varphi_\lambda)_\chi(0))-(\varphi_{\lambda'})_\chi(0)|^2dm(\chi)\\
&\lesssim |\lambda-\lambda'|^2 \int_{\ttinf} \big( |(\psi_0)_\chi(0)|+ |(\psi_1)_\chi(0)| \big)^2 dm(\chi)\\
&\lesssim |\lambda-\lambda'|^2
\end{align*}
since for all $\chi\in\ttinf$, $|(\psi_0)_\chi(0)|\leq C$ and $|(\psi_1)_\chi(0)|\leq C.$
\end{proof}

\subsection{Essential norm of the difference and connected components}

The aim of this subsection is to give a necessary condition for two composition operators to be in the same connected component. Inspired by \cite{ma}, we will deduce it from a minoration of the essential norm of the difference of two composition operators. Usually, this can been done using reproducing kernels. As pointed out in Section \ref{sec:rk}, reproducing kernels are less efficient for composition operators on $\mathcal H$ and they have been replaced by partial reproducing kernels
to study the compactness of composition operators with polynomial symbols.
However, Proposition 5 of \cite{b2} is not sufficient in our context, even if we intend to prove
that two composition operators with polynomial symbols are not in the same component: an arc between these two operators could involve non polynomial symbols.

Therefore, we need supplementary material about partial reproducting kernels.
We fix $\mathcal Q$ a finite subset of $\mathcal P$, the set of prime numbers.
Let $\mathbb N_{\mathcal Q}$ be the set of integers which can be factorized using prime numbers in $\mathcal Q$.
For $\varphi(s)=c_0s+c_1+\sum_{n\geq 1}c_nn^{-s}$, we let
$$\varphi_{\mathcal Q}(s)=c_0s+c_1+\sum_{n\in\mathbb N_{\mathcal Q}} c_n n^{-s}=: c_0 s+\psi_{\mathcal Q}(s).$$
We first observe that $\psi_{\mathcal Q}$,
which is well-defined on $\CC_{1/2}$, where $\psi_{\mathcal Q}$ converges absolutely, extends to $\CC_+$,
and that $\psi_{\mathcal Q}(\CC_+)\subseteq\overline{\CC_+}$.
Indeed, let us write any $z\in\mathbb T^\infty$ as $z=(z_\mathcal Q,w)$ where $z_\mathcal Q=(z_{i_1},\dots,z_{i_d})$, $\mathcal Q=\{p_{i_1},\dots,p_{i_d}\}$.
We may define $F$ a.e. on $\TT^d$ by
\[ F(z_{\mathcal Q})=\int_{\TT^\infty}\frac{\mathcal B\psi(z_{\mathcal Q},w)-1}{\mathcal B\psi(z_{\mathcal Q},w)+1}dm(w) \]
which clearly belongs to $H^{\infty}(\TT^d)$ with $\|F\|_\infty\leq 1$. By looking at the Taylor coefficients, $F=(\mathcal B \psi_{\mathcal Q}-1)/(\mathcal B \psi_{\mathcal Q}+1)$ in a
neighbourhood of zero. Therefore $\mathcal B\psi_{\mathcal Q}$ extends holomorphically to $\DD^d$ with $\A (\mathcal B\psi_{\mathcal Q})\geq 0$.
This yields that $\psi_{\mathcal Q}$ extends to $\CC_+$ with $\psi_{\mathcal Q}(\CC_+)\subseteq \overline{\CC_+}$.

By the open mapping theorem, $\psi_{\mathcal Q}(\CC_+)\subseteq \CC_+$ excepts if $\psi_{\mathcal Q}$ is constant.
Obviously when $c_0=0$, if $\psi_{\mathcal Q}$ is constant, then
$\psi_{\mathcal Q}=c_1$ and $\A(c_1)>1/2$. Hence, in all cases,
for all $\mathcal Q\subseteq\mathcal P$,
$\varphi_{\mathcal Q}\in \mathcal G$.

Let us also introduce a variant of partial reproducing kernels. We set, for $w\in\CC_+$,
$$K_{\mathcal Q,w}(s)=\prod_{p\in\mathcal Q}\frac{1}{1-p^{-(\bar w+s)}}$$
which is a well-defined function on $\CC_+$, belonging to $\mathcal H$, with
$$\|K_{\mathcal Q,w}\|^2=\prod_{p\in\mathcal Q}\frac 1{1-p^{-2\A w}}.$$
Moreover, it is easy to check for any sequence $(s_k)\subseteq\CC_+$ with $\A(s_k)\to 0$, $K_{\mathcal Q,s_k}/\|K_{\mathcal Q,s_k}\|$ converges weakly to zero.

The following lemma links $C_\varphi$, $\varphi_\mathcal Q$ and $K_{\mathcal Q,w}$. Its proof  follows exactly that of Proposition 5 in \cite{b2}.
\begin{lemma}\label{lem:prk}
Let $\varphi\in\mathcal G$ and let $\mathcal Q\subseteq\mathcal P$ finite. Then, for all $w\in\CC_+$,
\[ C_\varphi^*(K_{\mathcal Q,w})=
\left\{
\begin{array}{ll}
K_{\mathcal Q,\varphi_{\mathcal Q}(w)}&\textrm{ provided }c_0\geq 1\\
K_{\varphi_{\mathcal Q}(w)}&\textrm{ provided }c_0=0.\\
\end{array}\right. \]
\end{lemma}

Let us finally recall the Julia-Caratheodory theorem on the right half-plane $\CC_+$.
\begin{definition}
Let $\varphi:\CC_+\to\CC_+$ be analytic and let $\alpha\in\mathbb R$. We say that $\varphi$ admits
a {\bf finite angular derivative} at $i\alpha$ provided
\[ \liminf_{s\to i\alpha}\frac{\A \varphi(s)}{\A s}<+\infty .\]
\end{definition}

\begin{lemma}\label{lem:jc}
Let $\varphi:\CC_+\to\CC_+$ be analytic and let $\alpha\in\mathbb R$. Assume that $\varphi$ admits
a finite angular derivative at $i\alpha$ and let
\[ d:= \liminf_{s\to i\alpha}\frac{\A \varphi(s)}{\A s}<+\infty .\]
Then both $\varphi$ and $\varphi'$ have respective nontangential limits $\varphi(i\alpha)$
and $\varphi'(i\alpha)$, with $\varphi(i\alpha)\in i\mathbb R$ and $\varphi'(i\alpha)=d$. Moreover, $d\in (0,+\infty)$, $d$ is the nontangential limit of $\A\varphi (s)/\A s$ and $(\varphi(s)-\varphi(i\alpha))/(s-i\alpha)$  converges nontangentially to $\varphi'(i\alpha)$.
\end{lemma}

We are now ready to give a lower bound for the essential norm of the difference of two composition
operators (see \cite{ma} for the corresponding result on the disk).

\begin{theorem}\label{thm:essentialnormdifference}
Let $\varphi_0,\varphi_1\in\mathcal G$ with positive characteristic.
Let $\mathcal Q\subseteq\mathcal P$ finite, $d=\mathrm{card}(\mathcal Q)$.
Assume that there exists $\alpha\in i\mathbb R$ such that $\pzq$ admits a finite angular derivative at $i\alpha$.
Then either $\puq$ admits a finite angular derivative at $i\alpha$ with $\puq(i\alpha)=\pzq(i\alpha)$ and $\puq'(i\alpha)=\pzq'(i\alpha)$ or
\[ \|\cphiz-\cphiu \|_e\geq \frac 1{\pzq'(i\alpha)^{d/2}} .\]
\end{theorem}

\begin{proof}
Without loss of generality, we may and shall assume that $i\alpha=0$ and $\pzq(i\alpha)=0$.
We assume that either $\varphi_{1,\mathcal Q}$ does not admit a finite angular derivative at $0$ or
that either $\varphi_{1,\mathcal Q}(0)\neq 0$ or $\varphi_{1,\mathcal Q}'(0)\neq\varphi_{0,\mathcal Q}'(0)$.
For a sequence $(s_k)\subseteq \CC_+$ with $\A(s_k)$ tending to zero,
we set $k_{\mathcal Q,s_k}=K_{\mathcal Q,s_k}/\|K_{\mathcal Q,s_k}\|$ and observe that $k_{\mathcal Q,s_k}$ converges weakly to zero. Therefore, by Lemma \ref{lem:prk},
\begin{align*}
\|C_{\varphi_0}-C_{\varphi_1}\|_e^2&=\|\cphiz^*-\cphiu^*\|_e^2\\
&\geq \limsup_{k\to+\infty}\| \cphiz^*(k_{\mathcal Q,s_k})-\cphiu^*(k_{\mathcal Q,s_k})\|^2\\
&\geq \limsup_{k\to+\infty}\frac{\|K_{\mathcal Q,\pzq(s_k)}\|^2+\|K_{\mathcal Q,\puq(s_k)}\|^2-2\A K_{\mathcal Q,\pzq(s_k)}(\puq(s_k))}{\|K_{\mathcal Q,s_k}\|^2}.
\end{align*}
If $(s_k)$ converges nontangentially to zero, using the computation of the norm of the partial reproducing kernel and Julia-Caratheodory theorem, we get
\[ \frac{\|K_{\mathcal Q,\pzq(s_k)}\|^2 }{\|K_{\mathcal Q,s_k}\|^2}\sim \left(\frac{\A s_k}{\A \pzq(s_k)}\right)^{d}\sim \frac 1{\pzq'(0)^d}. \]
Hence, to conclude, we just need to find a sequence $(s_k)\subseteq \CC_+$ tending nontangentially to zero and such that
\begin{equation}\label{eq:essentialnormdifference}
\lim_{k\to+\infty}\frac{K_{\mathcal Q,\pzq(s_k)}(\puq(s_k))}{\|K_{\mathcal Q,s_k}\|^2}=
\lim_{k\to+\infty}\prod_{p\in\mathcal Q}\frac{1-p^{-2\A(s_k)}}{1-p^{-(\overline{\pzq(s_k)}+\puq(s_k))}}
\end{equation}
is as small as possible. First of all, if $\lim_{\sigma\to 0^+}\puq(\sigma)\neq 0$, we consider a sequence $(\sigma_n)$ of positive real numbers tending to zero and such that $|\puq(\sigma_n)|\geq \delta$ for some $\delta>0$. Then, setting $s_k=\sigma_k$,
it is clear that the limit appearing in \eqref{eq:essentialnormdifference} is equal to zero.
Thus we may assume that $\puq$ has radial limit $0$ at $0$.
Let $M>0$ and set $s_k=\frac 1k+i\frac Mk$. We observe that
\begin{align*}
\prod_{p\in\mathcal Q}\frac{1-p^{-2\A(s_k)}}{1-p^{-(\overline{\pzq(s_k)}+\puq(s_k))}} 	
&\sim_{+\infty}\left(\frac{\overline{\pzq(s_k)}+\puq(s_k)}{2\A s_k}\right)^{-d}\\
&\sim_{+\infty}\left(\frac{s_k}{2\A (s_k)}\right)^{-d}\times\left(\frac{\overline{\pzq(s_k)}}{\overline{s_k}}\times\frac{\overline{s_k}}{s_k}+\frac{\puq(s_k)}{s_k}\right)^{-d}.
\end{align*}
Provided $M$ is large enough, $|s_k|/\A(s_k)$ could become as large as we want when $k$ tends to $+\infty$. On the other hand,
\[ \frac{\overline{\pzq(s_k)}}{\overline{s_k}}\times\frac{\overline{s_k}}{s_k} \xrightarrow{k\to+\infty} \pzq'(0)\times\frac{1-iM}{1+iM}. \]
If $\puq$ does not have finite angular derivative at $0$, then $\puq(s_k)/s_k$ tends to $+\infty$. If $\puq$ has a finite angular derivative at $0$ with $\puq'(0)\neq \pzq'(0)$, then
\begin{align*}
\lim_{k\to+\infty} \left| \frac{\overline{\pzq(s_k)}}{\overline{s_k}}\times\frac{\overline{s_k}}{s_k}+\frac{\puq(s_k)}{s_k} \right| &={\left|\pzq'(0)\times\frac{1-iM}{1+iM}+\puq'(0)\right|}\\
&\geq \frac 12\left|\pzq'(0)-\puq'(0)\right|
\end{align*}
provided $M$ is sufficiently large. Therefore, in both cases, for any $\veps>0$, we have shown that we may choose $M>0$ such that
\[ \lim_{k\to+\infty} \left| \prod_{p\in\mathcal Q}\frac{1-p^{-2\A(s_k)}}{1-p^{-(\overline{\pzq(s_k)}+\puq(s_k))}} \right| \leq\veps. \]
This yields the result.
\end{proof}

\begin{remark}
Theorem \ref{thm:essentialnormdifference} is only interesting when $\varphi_0$ and $\varphi_1$
have unrestricted range. However, even if $\varphi_0$ is continuous, it could be possible that there does not exist $\alpha\in\RR$ with $\A\varphi(i\alpha)=0$. Nevertheless, even in that case, one can apply Theorem \ref{thm:essentialnormdifference} to some $\varphi_\chi$ as in the proof of the fothcoming Theorem \ref{thm:lcpolynomial}.
\end{remark}

\begin{corollary}\label{cor:necsamecomponent}
Let $\varphi_0,\ \varphi_1\in\mathcal G$ with $\cha(\varphi_0)=\cha(\varphi_1)>0$.
Assume that $\cphiz$ and $\cphiu$ belong to the same component of $\mathcal C(\mathcal H)$.
Then for all $\mathcal Q\subseteq\mathcal P$ finite and all $\alpha\in\RR$ such that $\pzq$ admits a finite angular derivative at $i\alpha$, $\puq$ admits a finite angular derivative at $i\alpha$ with $\pzq(i\alpha)=\puq(i\alpha)$ and $\pzq'(i\alpha)=\puq'(i\alpha)$.
\end{corollary}
\begin{proof}
Let $\mathcal Q\subseteq\mathcal P$ finite and let $\alpha\in\RR$ be such that $\pzq$ admits a finite angular derivative at $i\alpha$.  Assume that $\cphiu$ belongs to the connected component $\mathcal C_0$ of $\cphiz$ in $\mathcal C(\mathcal H)$. For $C_\tau\in\mathcal C_0$, let $U_\tau$ be the ball
(in $\mathcal C(\mathcal H)$) of center $C_\tau$ and radius $1/2\pzq'(i\alpha)^{d/2}$, with $d=\mathrm{card}(\mathcal Q)$. There exists a simple chain $U_{\tau_1},\dots,U_{\tau_n}$
extracted from $(U_\tau)_{C_\tau\in\mathcal C_0}$ with $\cphiz\in U_{\tau_1}$, $\cphiu\in U_{\tau_n}$ and $U_{\tau_i}\cap U_{\tau_j}\neq\varnothing$ for $i\neq j$ if and only if $|i-j|\leq 1$.
Let $C_{\tau_{i}}\in U_{\tau_i}\cap U_{\tau_{i+1}}$. Then setting $\tilde{\varphi_0}=\varphi_0$,
$\tilde{\varphi_i}=\tau_i$ for $i=1,\dots,n-1$, and $\tilde\varphi_n =\varphi_1$, we have
\[ \| C_{\tilde{\varphi_i}}-C_{\widetilde{\varphi_{i-1}}}\| < \frac1{\pzq'(i\alpha)^{d/2}}\]
for all $i=1,\dots,n$. By an immediate induction and Theorem \ref{thm:essentialnormdifference}, we get that for all $i=0,\dots,n$, $\tilde{\varphi_i}$
admits a finite angular derivative at $i\alpha$ with $\tilde{\varphi_i}(i\alpha)=\pzq(i\alpha)$ and
$\tilde{\varphi_i}'(i\alpha)=\pzq'(i\alpha)$. This is exactly the desired result for $i=n$.
\end{proof}

\begin{example}
Let $\varphi_0(s)=s+(1-2^{-s})$ and $\varphi_1(s)=s+(1-3^{-s})$.
Then $\cphiz$ and $\cphiu$ are not in the same component of $\mathcal C(\mathcal H)$.
\end{example}

We just need to apply Corollary \ref{cor:necsamecomponent} with $\mathcal Q=\{2\}$.

\begin{example}
Let $\varphi_0(s)=s$. Then $C_{\varphi_0}$ is isolated in $\mathcal C(\mathcal H)$.
\end{example}

Indeed, if $\cphiu$ belongs to the same component of $\cphiz$, then for all $\mathcal Q\subseteq\mathcal P$, for all $\alpha\in \RR$, $\puq$ admits an angular derivative at $i\alpha$
with $\puq(i\alpha)=\pzq(i\alpha)=\varphi_0(i\alpha)=i\alpha$. Hence, $\puq(s)=s$ which in turn implies $\varphi_1(s)=s$.

\medskip

The use of the angular derivative is not as efficient when $c_0=0$. The reason is simple: it is possible that $\varphi_{\mathcal Q}$ admits a finite angular derivative at $i\alpha$ and that $C_\varphi$ is compact. This is for instance the case of
\[ \varphi(s)=\frac 12+(1-2^{-s})+(1-3^{-s}). \]
Nevertheless, we still get a result when $\mathrm{card}(\mathcal Q)=1$. Observe that at a boundary point $i\alpha$, we now have $\varphi_{\mathcal Q}(i\alpha)\in \frac 12+i\RR$ and that we are now interested in $(\A\varphi_\mathcal Q(s)-1/2)/\A s$.

\begin{theorem}\label{thm:essentialnormdifferencezero}
Let $\varphi_0,\varphi_1\in\mathcal G$ with zero characteristic. Let $\mathcal Q\subseteq\mathcal P$ with $\mathrm{card}(\mathcal Q)=1$. Assume that there exists $\alpha\in i\mathbb R$ such that $\pzq$ admits a finite angular derivative at $i\alpha$. Then either $\puq$ admits a finite angular derivative at $i\alpha$ with $\puq(i\alpha)=\pzq(i\alpha)$ and $\puq'(i\alpha)=\pzq'(i\alpha)$ or
\[ \|\cphiz-\cphiu \|_e\geq \frac {C_\mathcal Q}{\pzq'(i\alpha)^{1/2}} .\]
\end{theorem}
\begin{proof}
We only  point out the differences with the proof of Theorem \ref{thm:essentialnormdifference}.
We have to replace everywhere $K_{\mathcal Q,\pzq(s_k)}$ by $K_{\pzq(s_k)}$. One now has
\[ \|K_w\|^2=\zeta(2\A (w)) \]
and
\[ K_{\pzq(s_k)}(\puq(s_k))=\zeta\left(\overline{\pzq(s_k)}+\puq(s_k)\right). \]
The estimate of this last term follows that done in Theorem \ref{thm:essentialnormdifference},
except that we have to use $\zeta(s)\sim 1/(s-1)$ if $s\to 1$ instead of $1-p^{-s}\sim -(\log p)s$ if $s\to 0$.
\end{proof}

\begin{corollary}\label{cor:necsamecomponentzero}
Let $\varphi_0,\ \varphi_1\in\mathcal G$ with $\cha(\varphi_0)=\cha(\varphi_1)=0$.
Assume that $\cphiz$ and $\cphiu$ belong to the same component of $\mathcal C(\mathcal H)$.
Then for all $\mathcal Q\subseteq\mathcal P$ with $\mathrm{card}(\mathcal Q)=1$ and all $\alpha\in\RR$ such that $\pzq$ admits a finite angular derivative at $i\alpha$, $\puq$ admits a finite angular derivative at $i\alpha$ with $\pzq(i\alpha)=\puq(i\alpha)$ and $\pzq'(i\alpha)=\puq'(i\alpha)$.
\end{corollary}

\begin{example}
Let $\varphi_0(s)=\frac 12+(1-2^{-s})$ and $\varphi_1(s)=\frac 12+(1-3^{-s})$. Then $\cphiz$ and $\cphiu$ do not belong to the same component of $\mathcal C(\mathcal H)$.
\end{example}

\medskip

Corollary \ref{cor:necsamecomponent} already gives interesting examples of composition operators that do not belong to the same component of $\mathcal C(\mathcal H)$. However, it does not help us to solve the following simple problem: does $\cphiz$ and $\cphiu$ belong to the same component of $\mathcal C(\mathcal H)$ where
\begin{align*}
\varphi_0(s)&=s+(2\log 2+\log 3)-\log 3\cdot 2^{-s}-2\log 2\cdot 3^{-s}\\
\varphi_1(s)&=s+(\log 2+2\log 3)-2\log 3\cdot 2^{-s}-\log 2\cdot 3^{-s}?
\end{align*}
A plausible conjecture would be that Corollary \ref{cor:necsamecomponent} could be extended to the angular partial derivatives of the Bohr lifts $\mathcal B\psi_{0,\mathcal Q}$ and $\mathcal B\psi_{1,\mathcal Q}$. Nevertheless, the curve $\sigma\mapsto (2^{-\sigma},\dots,p_d^{-\sigma})$ is not a special curve in $\mathbb D^d$, which prevents the use of the Julia-Caratheodory theorem of the polydisc (see \cite{Ab98}).

However, inspecting the partial derivatives of the Bohr lifts will be fruitful for studying compact linear combinations of composition operators with regular symbols, as indicated by the results of the next section.


\section{Compact linear combinations of polynomial symbols}
\label{sec:clcps}
\subsection{Lower bound for the esssential norm from second order data}
We start with a result saying that if a linear combination $\lambda_1 C_{\varphi_1}+\cdots+\lambda_N C_{\varphi_N}$ is compact,
then the coefficients $\lambda_j$ should satisfy some properties depending on the behaviour of the $\mathcal B\psi_j$ at their boundary points. We need
first to introduce some terminology inspired by \cite{km}.

\begin{definition} Let $F,G:\overline{\mathbb D^d}\to \overline{\mathbb C_+}$ be of class ${\mathcal C^2}$ and let $\xi \in \TT^d$.
We say that $F$ and $G$ have the \textbf{same second order boundary data} at $\xi$ if $F(\xi)=G(\xi)$ and $\partial_\alpha F(\xi)=\partial_\alpha G(\xi)$
for all $\alpha\in\mathbb N^d$ with $|\alpha|\leq 2$. We also denote by $ \Gamma(F)=\{z\in\TT^d:\ \A F(z)=0\}$.
\end{definition}

We will need a couple of lemmas.

\begin{lemma}\label{lem:severalquadratic}
Let $d\geq 1$ and let $Q$ be a non-zero quadratic polynomial in $d$ variables (with complex coefficients). Then $Z(Q)=\{x\in\mathbb R^d:\ Q(x)=0\}$ has Lebesgue measure zero.
\end{lemma}

\begin{proof}
The proof is done by induction on $d$. For $d=1$, the result is of course trivial. Assume that the result has been done until $d$ and let us consider $Q$ a quadratic polynomial in $d+1$ variables. Fix $x_{d+1}\in\mathbb R$ and let $Q_{x_{d+1}}(x_1,\dots,x_d)=Q(x_1,\dots,x_{d+1})$.
Then either $Q_{x_{d+1}}$ is identically zero or its zero set has Lebesgue measure zero. Writing
$$Q(x)=a+\sum_{l=1}^{d+1}b_lx_l+\sum_{1\leq l\leq m\leq d+1}c_{l,m}x_lx_m,$$
the former case implies that
$$\left\{
\begin{array}{rcl}
a+b_{d+1}x_{d+1}+c_{d+1,d+1}x_{d+1}^2&=&0\\
\forall l=1,\dots, d,\ b_l+c_{l,d+1}x_{d+1}&=&0.
\end{array}\right.$$
Since $Q$ itself is assumed to be non-zero, we can suppose, permuting the variables if necessary, that either $a\neq 0$ or $b_{d+1}\neq 0$ or $c_{l,d+1}\neq 0$ for some $l=1,\dots,d+1$. Therefore there are at most two values of $x_{d+1}$ such that $Q_{x_{d+1}}$ is identically zero. Hence, for almost all $x_{d+1}\in\mathbb R$,
$$\int_{\RR^d}\mathbf 1_{Z(Q)}(x_1,\dots,x_{d+1})dx_1\dots dx_{d}=0.$$
The lemma follows from an application of Fubini's theorem.
\end{proof}

Our second lemma is a result about  functions which vanish along $(0,+\infty)$.

\begin{lemma}\label{lem:vanishingquadratic}
Let $\mathcal C\subseteq\mathbb N$ be finite, let $\alpha>0$, let $(\lambda_j)_{j\in\mathcal C}$ be complex numbers,
let $(\mu_j)_{j\in\mathcal C}$ be pairwise distinct complex numbers and {let $d\geq 1$.}
Assume that, for all $M>0$,
\[ \sum_{j,j'\in\mathcal C}\frac{\lambda_j \overline{\lambda_{j'}}}{(M\alpha+\mu_j+\overline{\mu_{j'}})^d}=0. \]
Then $\lambda_j=0$ for all $j\in\mathcal C$.
\end{lemma}
\begin{proof}
Let us consider
$$F(z)=\sum_{j,j'\in\mathcal C}\frac{\lambda_j \overline{\lambda_{j'}}}{(\alpha+z(\mu_j+\overline{\mu_{j'}}))^d}.$$
Then $F$ is an analytic function defined in a neighbourhood of the origin and the assumption tells us that the origin is an accumulation point of the zeros of $F$. The result now follows from \cite[Corollary 3.4]{ckw15}.
\end{proof}

\begin{theorem}\label{thm:essentialnormcomposition}
 Let 
 $\varphi_j(s)=c_{0,j}s+\psi_j$ be a finite sequence of symbols in $\mathcal G$ with $\psi_j$ involving only the first $d$ prime numbers, $j=1,\dots,N$.
 Assume that each $\mathcal B\psi_j$ extends to a $\mathcal C^2$-function on $\overline{\mathbb D^d}$ and that $\cha(\varphi_1)>0$. Let $(\lambda_j)_{j=1,\dots,N}$
 be a finite sequence of complex numbers such that $\sum_{j=1}^N \lambda_j C_{\varphi_j}$ is compact. Let $z\in\Gamma(\mathcal B\psi_1)$ and let $\mathcal J$ be the set of $j\in\{1,\dots,N\}$
 such that $\mathcal B\psi_j$ and $\mathcal B\psi_1$ have the same second order boundary data at $z$.
 Then
 $$\sum_{j\in\mathcal J}\lambda_j=0.$$
\end{theorem}

\begin{proof}
 Without loss of generality, we may and shall assume that $z=\mathbf e:=(1,\dots,1)$ and that $\mathcal B\psi_1(\mathbf e)=0$. Let us write
for $j=1,\dots,N$,
\begin{align*}
 \mathcal B\psi_j(z)&=A^{(j)}+\sum_{l=1}^d B_l^{(j)}(z_l-1)+\sum_{1\leq l\leq m\leq d}C_{l,m}^{(j)}(z_l-1)(z_m-1)+o(\|z-\mathbf e\|^2).
\end{align*}
It is a standard fact (see \cite[Lemma 7]{bb} that each $B_l^{(j)}\in (-\infty,0]$.
We set $c_0=c_{0,1}$ and introduce the following subsets of $\{1,\dots,N\}$:
\begin{align*}
I&=\left\{j:\ c_{0,j}=c_0\textrm{ and }\varphi_j(0)=0\right\}=\left\{j:\ c_{0,j}=c_0\textrm{ and }A^{(j)}=0\right\}\\
J_1&=\left\{j:\ c_{0,j}=c_0\textrm{ and }\varphi_j(0)\neq 0\right\}\\
J_2&=\left\{j:\ c_{0,j}\neq c_0\textrm{ and }c_{0,j}>0\right\}=:\{j_1,\dots,j_p\}\\
J_3&=\left\{j:\ c_{0,j}=0\right\}.
\end{align*}
For $j\in \{1,\dots,N\}$, let us introduce
\begin{align*}
L_j(x_1,\dots,x_d)&=\sum_{l=1}^d B_l^{(j)}\log(p_l) x_l\\
Q_j(x_1,\dots,x_d)&=\frac 12\sum_{l=1}^d B_l^{(j)}(\log p_l)^2 x_l^2+
\sum_{1\leq l\leq  m\leq d}C_{l,m}^{(j)}\log(p_l)\log(p_m)x_lx_m.
\end{align*}
We define an equivalence relation on $I$ by $j\sim j'\iff L_j=L_{j'}$. In other words, $j\sim j'$ if and only if $\mathcal B\psi_j$ and $\mathcal B\psi_{j'}$ have the same first order boundary data at $\mathbf e$. Denote by $\mathcal C_1,\dots,\mathcal C_q$ the equivalence classes for this relation with $1\in\mathcal C_1$.
We define another equivalence relation on $\mathcal C_1$ by $j\mathcal R j'\iff Q_j=Q_{j'}$ and we denote by $\mathcal C_{1,1},\dots,\mathcal C_{1,q'}$
the equivalence classes for this relation, with $1\in\mathcal C_{1,1}$. In other words, $j\mathcal R j'$ if and only if $\mathcal B\psi_j$ and $\mathcal B\psi_{j'}$ have the same second order boundary data. In particular, $\mathcal J$ is nothing else than $\mathcal C_{1,1}$.
For $i\neq i'\in\{1,\dots,q\}$, let
$$E_{i,i'}=\left\{\beta\in\RR^d:\ \forall j\in\mathcal C_i,\ \forall j'\in\mathcal C_{i'},\ L_j(\beta)\neq L_{j'}(\beta)\right\}$$
(we can replace $\forall$ by $\exists$ in the previous definition). Similarly for  $i\neq i'\in\{1,\dots,q'\}$, let
$$F_{i,i'}=\left\{\beta\in\mathbb R^d:\ \forall j\in\mathcal C_{1,i},\ \forall j'\in\mathcal C_{1,i'},\ Q_j(\beta)\neq Q_{j'}(\beta)\right\}.$$
By Lemma \ref{lem:severalquadratic}, there exists $\beta\in \bigcap_{i\neq i'}E_{i,i'}\cap \bigcap_{i\neq i'}F_{i,i'}$. Let finally, for $1\leq l\leq d$, $k\geq 1$ and $M>0$,
$$z_l(k)=p_l^{-\frac Mk+i\frac{\beta_l}{\sqrt k}}$$
so that
$$z_l(k)=1+\frac{i\log(p_l)\beta_l}{\sqrt k}-\frac{2M\log(p_l)+\log^2(p_l)\beta_l^2}{2k}+o\left(\frac 1k\right).$$
Therefore,
\begin{align*}
\mathcal B\psi_j(z(k))&=A^{(j)}+\frac{i}{\sqrt k}\sum_{l=1}^d B_l^{(j)}\log(p_l)\beta_l-\frac 1k\Bigg(M\sum_{l=1}^d B_l^{(j)}\log(p_l)\\
&\quad\quad+\frac12\sum_{l=1}^d B_l^{(j)} \log^2(p_l)\beta_l^2+\sum_{1\leq l\leq m\leq d}C_{l,m}^{(j)}\log(p_l)\log(p_m)\beta_l\beta_m\Bigg)+o\left(\frac 1k\right)\\
&=A^{(j)}+\frac i{\sqrt k}L_j(\beta_1,\dots,\beta_d)
-\frac1k\left(M L_j(\mathbf e)+Q_j(\beta_1,\dots,\beta_d)\right)+o\left(\frac 1k\right).
\end{align*}

We now choose a sequence of real numbers $(t_k)$ such that
\begin{enumerate}
\item for $1\leq l\leq d$, $p_l^{-it_k}$ is very close to $p_l^{\frac{i\beta_l}{\sqrt k}}$, so close that, setting $w_k=\frac{M}k+it_k$,
\[\mathcal B\psi_j(p_1^{-w_k},\dots,p_d^{-w_k})=\mathcal B\psi_j(z(k))+o\left(\frac 1k\right)\]
for all $j=1,\dots,N$.
\item For all $j\in J_2$, writing $j=j_r$ with $1\leq r\leq p$,
$ p_{d+r}^{-it_k(c_{0,j_r}-c_0)}p_{d+r}^{-A^{(j_r)}}$ tends to some $\omega(j_r)\neq 1$.
\end{enumerate}
These two conditions may be simultaneously satisfied because of Kronecker's theorem and since $c_{0,j}\neq c_0$ for all $j\in J_2$. We set $d'=d+p$ and use that
\begin{align*}
\left\|\sum_{j=1}^N \lambda_j C_{\varphi_j}\right\|_e&\geq \limsup_{k\to+\infty}
\left\|\sum_{j=1}^N \overline{\lambda_j} C_{\varphi_j}^*(k_{d',w_k})\right\|^2\\
&\geq \limsup_{k\to+\infty}\sum_{1\leq j,j'\leq N}\lambda_j \overline{\lambda_{j'}} \langle C_{\varphi_{j'}}^*(k_{d',w_k}),C_{\varphi_{j}}^*(k_{d',w_k})\rangle.
\end{align*}
We claim that if $j\in I$ and $j'\notin I$, then
$$\langle C_{\varphi_{j'}}^*(k_{d',w_k}),C_{\varphi_{j}}^*(k_{d',w_k})\rangle\xrightarrow{k\to+\infty}0. $$
First assume that $j'\in J_3$. Then
\begin{align*}
 \langle C_{\varphi_{j'}}^*(k_{d',w_k}),C_{\varphi_{j}}^*(k_{d',w_k})\rangle&=\frac{\langle K_{\varphi_{j'}(w_k)},K_{d',\varphi_j(w_k)}\rangle}{\|K_{d',w_k}\|^2}   \\
 &=\frac{\langle K_{d',\varphi_{j'}(w_k)},K_{d',\varphi_j(w_k)}\rangle}{\|K_{d',w_k}\|^2}
\end{align*}
and this goes to zero since $\|K_{d',\varphi_{j'}(w_k)}\|$ is bounded because $\A \varphi_{j'}(w_k)\geq 1/2$.
If $j\in I$ and $j'\notin J_3$, then
$$ \langle C_{\varphi_{j'}}^*(k_{d',w_k}),C_{\varphi_{j}}^*(k_{d',w_k})\rangle = \prod_{l=1}^{d'}\frac{1-p_l^{-2\A(w_k)}}{1-p_l^{-(\overline{\varphi_{j'}(w_k)}+\varphi_j(w_k))}}.$$
Pick $j'\in J_1$. Then
\[ \overline{\varphi_{j'}(w_k)}+\varphi_j(w_k)\to \overline{A^{(j')}}\neq 0 \]
so that
\begin{equation}\label{eq:quadraticto0}
\prod_{l=1}^{d'}\frac{1-p_l^{-2\A(w_k)}}{1-p_l^{-(\overline{\varphi_{j'}(w_k)}+\varphi_j(w_k))}}\xrightarrow{k\to+\infty}0.
\end{equation}
Pick finally $j'\in J_2$ and write $j'=j_r$ for $1\leq r\leq p$. Then
$$\overline{\varphi_{j'}(w_k)}+\varphi_j(w_k)=(c_0-c_{0,j_r})it_k+A^{(j_r)}+o(1)$$
so that
\[ p_{d+r}^{-(\varphi_{j'}(w_k)+\varphi_j(w_k)}\xrightarrow{k\to+\infty}\omega(j')\neq 1. \]
Therefore,
\[ \frac{1-p_{d+r}^{-2\A(w_k)}}{1-p_{d+r}^{-(\overline{\varphi_{j'}(w_k)}+\varphi_j(w_k))}}
\xrightarrow{k\to+\infty}0. \]
On the other hand, since $\A \varphi_j(w_k),\A\varphi_{j'}(w_k)\geq \A(w_k)$, we also have
\[ \prod_{1\leq l\leq d',\ l\neq d+r}^{d'}\left| \frac{1-p_l^{-2\A(w_k)}}{1-p_l^{-(\overline{\varphi_{j'}(w_k)}+\varphi_j(w_k))}}\right| \leq 1. \]
Hence, \eqref{eq:quadraticto0} remains true for these values of $j$ and $j'$. This implies that
\begin{align*}
0\geq \left\|\sum_{j=1}^N \lambda_j C_{\varphi_j}\right\|_e&\geq \limsup_{k\to+\infty}
\Bigg( \sum_{j,j'\in I}\lambda_j\overline{\lambda_{j'}}\prod_{l=1}^{d'}\frac{1-p_l^{-2\A(w_k)}}{1-p_l^{-(\overline{\varphi_{j'}(w_k)}+\varphi_j(w_k))}}\\
&\quad\quad\quad\quad+ \sum_{j,j'\notin I}\lambda_j\overline{\lambda_{j'}}\prod_{l=1}^{d'}\frac{1-p_l^{-2\A(w_k)}}{1-p_l^{-(\overline{\varphi_{j'}(w_k)}+\varphi_j(w_k))}}\Bigg).
\end{align*}
Since each term in the previous sum is nonnegative (it is equal to $\left\|\sum_{j\in E}\lambda_j C_{\varphi_j}^*(k_{d,w_k})\right\|^2$ for either $E=I$ or $E=\{1,\dots,N\}\backslash I$),
this yields
\[ \limsup_{k\to+\infty}\sum_{j,j'\in I}\lambda_j\overline{\lambda_{j'}}\prod_{l=1}^{d'}\frac{1-p_l^{-2\A(w_k)}}{1-p_l^{-(\overline{\varphi_{j'}(w_k)}+\varphi_j(w_k))}}=0. \]
Consider now $j,j'\in I$, and observe that
\begin{align*}
\overline{\varphi_{j'}(w_k)}+\varphi_j(w_k)&=
\frac{2c_0 M}k+\overline{\mathcal B\psi_{j'}(z(k))}+\mathcal B\psi_j(z(k))+o\left(\frac 1k\right)\\
&=\frac{i(L_j(\beta)-L_{j'}(\beta))}{\sqrt k}+\frac{2c_0M}k-\frac{M(L_j(\mathbf e)+L_{j'}(\mathbf e))}{k}\\
&\quad\quad-\frac{\overline{Q_{j'}(\beta)}+Q_j(\beta)}{k}+o\left(\frac 1k\right).
\end{align*}
If $j\in\mathcal C_i$ and $j'\in\mathcal C_{i'}$ with $i\neq i'$, so that $L_j(\beta)\neq L_{j'}(\beta)$, then
\[ \overline{\varphi_{j'}(w_k)}+\varphi_j(w_k)\sim_{k\to+\infty} \frac{i(L_j(\beta)-L_{j'}(\beta))}{\sqrt k} \]
and this implies easily that, for all $l\geq 1$,
\[ \frac{1-p_l^{-2\A(w_k)}}{1-p_l^{-(\overline{\varphi_{j'}(w_k)}+\varphi_j(w_k))}}\xrightarrow{k\to+\infty}0. \]
If $j\neq j'\in\mathcal C_1$, one now gets that
\[ \frac{1-p_l^{-2\A(w_k)}}{1-p_l^{-(\overline{\varphi_{j'}(w_k)}+\varphi_j(w_k))}}\xrightarrow{k\to+\infty}\frac{M}{2M(c_0-L_1(\mathbf e))-(\overline{Q_{j'}(\beta)}+Q_j(\beta))}. \]
Arguing as above, we finally find that, for all $M>0$,
\begin{equation}\label{eq:essentialnormcomposition}
 \sum_{j,j'\in\mathcal C_1}\frac{\lambda_j\overline{\lambda_{j'}}M^{d'}}{(2M(c_0-L_1(\mathbf e))-{(\overline{Q_{j'}(\beta)}+Q_j(\beta) } )^{d'}}=0.
 \end{equation}
Let for $k=1,\dots,q'$, $\nu_k=\sum_{j\in\mathcal C_{1,k}}\lambda_j$ and let $\mu_k=Q_j(\beta)$ for some (all) $j\in\mathcal C_{1,k}$, $1\leq k\leq q'$. Then \eqref{eq:essentialnormcomposition}
may be rewritten
$$\sum_{k,k'=1}^{q'}\frac{\nu_k \overline{\nu_{k'}} M^{d'}}{\big(2M(c_0-L_1(\mathbf e))-(\overline{\mu_{k'}}+\mu_k)\big)^{d'}}=0.$$
Lemma \ref{lem:vanishingquadratic} tells us that $\nu_k=0$ for all $k=1,\dots,q'$. The case $k=1$ is exactly the required statement.
\end{proof}

\subsection{Linear combinations of polynomial symbols with degree at most 2}
We are now ready for the proof of our extension of the main result of \cite{bwy}.
\begin{theorem}\label{thm:lcpolynomial}
Let $(\varphi_j)_{j=1,\dots,N}\subseteq\mathcal G$ be a finite sequence of distinct symbols.
Assume that each $\varphi_j$ is a polynomial symbol with degree at most $2$.
Let also $(\lambda_j)_{j=1,\dots,N}$ be a finite sequence of non-zero complex numbers.
Then $\sum_{j=1}^N \lambda_j C_{\varphi_j}$ is compact on $\mathcal H$ if and only if each $C_{\varphi_j}$ is compact.
\end{theorem}
\begin{proof}
 We proceed by induction on $N$, the case $N=1$ being trivial. Assume that the result has been obtained for $N-1$ symbols and let us prove
 it for $N$ symbols. By contradiction, assume that no $C_{\varphi_j}$ is compact. If one symbol, say $\varphi_1$, has a positive characteristic,
 then we can conclude immediately using Theorem \ref{thm:essentialnormcomposition}. Indeed, since $\varphi_1$ has unrestricted range (it is assumed
 to be noncompact), $\Gamma(\mathcal B\psi_1)\neq\varnothing$ (the polycircle $\mathbb T^d$ is compact). Now, for $j\neq 1$, $\mathcal B\psi_1$ and $\mathcal B\psi_j$ cannot have the same second order
 boundary data at some point, since $\varphi_1$ and $\varphi_j$ are distinct polynomials of degree at most $2$. Thus, we get $\lambda_1=0$, a contradiction.

 Therefore, we are reduced to the case where all symbols $\varphi_j$ have zero characteristic. Let us fix $j\in\{1,\dots,N\}$. In the terminology of \cite{bb}, the complex dimension of $\varphi_j$ cannot
 be greater than or equal to $2$: since the degree of $\varphi_j$ is at most 2, Theorem 3 of \cite{bb} would imply that $C_{\varphi_j}$ is compact. Thus we have two possibilities for $\varphi_j$: either
 \begin{equation}\label{eq:lcpol1}
 \varphi_j(s)=c_{1,j}+d_j q_{1,j}^{-s}q_{2,j}^{-s}\textrm{ with }q_{1,j}\neq q_{2,j}\in\mathcal P,\ d_j\neq 0
 \end{equation}
 which corresponds to the case where the complex dimension and the degree of $\varphi_2$ are equal to 2,
 or
 \begin{equation}\label{eq:lcpol2}
 \varphi_j(s)=c_{1,j}+e_j q_j^{-s}+f_j q_j^{-2s}\textrm{ with }q_j\in\mathcal P
 \end{equation}
 which means that $\varphi_j$ has complex dimension 1 and degree 1 or 2 (provided $f_j=0$ or not).
If all symbols may be written as in \eqref{eq:lcpol1}, then we face to linear symbols and this case will be handled in the forthcoming Theorem \ref{thm:lclinearzero}. Therefore we assume that at least one symbol, say $\varphi_1$, has a form like in \eqref{eq:lcpol2}. Let $\chi\in\ttinf$. Since $(f\circ\varphi_j)_\chi=f\circ (\varphi_j)_\chi$, the following diagram commutes
$$\CD
  \mathcal{H} @> \sum_{j}\lambda_{j}C_{\varphi_{j}}  >>  \mathcal{H} \\
  U @V VV @V V VV  \\
   \mathcal{H} @> \sum_{j}\lambda_{j}C_{(\varphi_{j})_{\chi}} >> \mathcal{H}
\endCD $$
where $U$ and $V$ are the unitary maps of $\mathcal H$ defined by $U(f)=f$ and $V(f)=f_\chi$.
Therefore, the compactness of $\sum\lambda_j C_{\varphi_j}$ is equivalent to that of
$ \sum_{j}\lambda_{j}C_{(\varphi_{j})_{\chi}} $ which implies that we may assume that $\A (\varphi_1(0))=1/2$.
By vertical translations, we may even assume that $\varphi_1(0)=1/2$. Let
\begin{align*}
I&=\big\{j:\ \varphi_j\textrm{ may be written as in \eqref{eq:lcpol2} with }q_j=q_1,\ 
{\varphi_j(0)=\frac 12,}\ \varphi_j'(0)=\varphi_1'(0)\Big\}\\
J_1&=\Big\{j:\ \varphi_j\textrm{ may be written as in \eqref{eq:lcpol2} with }q_j=q_1,\
{\varphi_j(0)=\frac 12,}\ \varphi_j'(0)\neq \varphi_1'(0)\Big\}\\
J_2&=\Big\{j:\ \varphi_j\textrm{ may be written as in \eqref{eq:lcpol2} with }q_j=q_1,\
{\varphi_j(0)\neq \frac 12} \Big\}\\
J_3&=\Big\{j:\ \varphi_j\textrm{ may be written as in \eqref{eq:lcpol2} with }q_j\neq q_1\Big\}\\
J_4&=\Big\{j:\ \varphi_j\textrm{ may be written as in \eqref{eq:lcpol1}}\Big\}.
\end{align*}
Let finally $s_k=\frac 1k+i\frac M{\sqrt k}$ for some $M>0$ and consider the normalized partial reproducing kernels $k_{s_k}=K_{s_k}^{(q_1)}/\|K_{s_k}^{(q_1)}\|$; the sequence $(k_{s_k})$ goes weakly to zero. By Lemma \ref{lem:prk}, if $j\in I\cup J_1\cup J_2$,
\[ C_{\varphi_j}^*(k_{s_k})=\frac{K_{\varphi_j(s_k)}}{\|K_{s_k}^{(q_1)}\|} \]
whereas, if $j\in J_3\cup J_4$,
\[  C_{\varphi_j}^*(k_{s_k})=\frac{K_{c_{1,j}}}{\|K_{s_k}^{(q_1)}\|}\xrightarrow{k\to+\infty}0. \]
Using this last result, one gets
\begin{align*}
\left\|\sum_{j=1}^N \lambda_j C_{\varphi_j}\right\|_e^2&\geq \limsup_{k\to+\infty}
\left\|\sum_{j=1}^N \overline{\lambda_j} C_{\varphi_j}^*(k_{s_k})\right\|^2\\
&\geq \limsup_{k\to+\infty}\sum_{j,j'\in I\cup J_1\cup J_2}\frac{\overline{\lambda_j } {\lambda_{j'}}
\langle K_{\varphi_j(s_k)},K_{\varphi_{j'}(s_k)}\rangle}{\|K_{s_k}^{(q_1)}\|^2}\\
&\geq \limsup_{k\to+\infty}\sum_{j,j'\in I\cup J_1\cup J_2} \frac{\overline{\lambda_j} {\lambda_{j'}}
\zeta (\overline{\varphi_{j'}(s_k)}+\varphi_j(s_k))}{\|K_{s_k}^{(q_1)}\|^2}.
\end{align*}
We claim that, for all $j\in I$ and all $j'\in J_1\cup J_2$,
\begin{equation}\label{eq:lcpol3}
 \frac{
\zeta (\overline{\varphi_{j'}(s_k)}+\varphi_j(s_k))}{\|K_{s_k}^{(q_1)}\|^2}\xrightarrow{k\to+\infty}0.
\end{equation}
Let us observe that, for all $j\in\{1,\dots,N\}$,
$${\varphi_j(s_k)=\varphi_j(0)}+\frac{iM \varphi_j'(0)}{\sqrt k}+\frac{2\varphi_j'(0)-M^2\varphi_j''(0)}{2k}+
o\left(\frac 1k\right) $$
so that if $j\in I$ and $j'\in I\cup J_1\cup J_2$,
\begin{align*}
\overline{\varphi_{j'}(s_k)}+\varphi_j(s_k)&=\frac 12+{\overline{\varphi_{j'}(0)}+i\frac{\varphi'_j(0)-\varphi'_{j'}(0)}{\sqrt k}M } \\
&\quad\quad\quad+\frac{2\varphi_j'(0)+{2\varphi_{j'}'(0)}-M^2(\overline{\varphi_{j'}''(0)}+\varphi_j''(0))}{2k}+o\left(\frac 1k\right).
\end{align*}
Hence, \eqref{eq:lcpol3} follows now easily from the properties of $\zeta$ and $\|K_{s_k}^{(q_1)}\|^2\sim Ak$. On the other hand, when $j,j'\in I$, one gets
\[ \frac{
\zeta (\overline{\varphi_{j'}(s_k)}+\varphi_j(s_k))}{\|K_{s_k}^{(q_1)}\|^2}\xrightarrow{k\to+\infty} \frac{2A}{4\varphi_1'(0)-M^2(\overline{\varphi_{j'}''(0)}+\varphi_j''(0))}. \]
Therefore, arguing as in the proof of Theorem \ref{thm:essentialnormcomposition}, we get that for all $M>0$,
\[\sum_{j,j'\in I} \frac{\lambda_j \overline{\lambda_{j'}}}{4\varphi_1'(0)-M^2(\overline{\varphi_{j'}''(0)}+\varphi_j''(0))}=0. \]
Now, for $j\neq j'\in I$, one has $\varphi_j''(0)\neq\varphi_{j'}''(0)$. Indeed, if $j,j'\in I$ with $\varphi_j''(0)=\varphi_{j'}''(0)$, then
\begin{align*}
\varphi_j'(0)&=-(\log q_1) e_j-2(\log q_1) f_j=-(\log q_1) e_{j'}-2(\log q_1) f_{j'}=\varphi_{j'}(0)\\
\varphi_j''(0)&=(\log q_1)^2 e_j+4(\log q_1)^2 f_j=(\log q_1)^2 e_{j'}+4(\log q_1)^2 f_{j'}=\varphi_{j'}''(0).
\end{align*}
{which implies $e_j=e_{j'}$ and $f_j=f_{j'}$,} thus $\varphi_j=\varphi_{j'}$ and $j=j'$.
Finally, Lemma \ref{lem:vanishingquadratic} gives $\lambda_j=0$ for $j\in I$, a contradiction.
\end{proof}

\subsection{Linear combinations of linear symbols}
Theorem \ref{thm:lcpolynomial} does not encompass completely the main result of \cite{bwy}. Indeed, for instance,
$\varphi(s)=s+1-30^{-s}$ is a linear symbol which is not a polynomial symbol of degree at most $2$.
Nevertheless, we are also able to handle these symbols, thanks to the two forthcoming theorems.

\begin{theorem}\label{thm:lclinear}
Let $(\varphi_j)_{j=1,\dots,N}\subseteq\mathcal G$ be a finite sequence of distinct linear symbols such that at least one of them has a positive characteristic.
Let $(\lambda_j)_{j=1,\dots,N}$ be a finite sequence of non-zero complex numbers. Then $\sum_{j=1}^N \lambda_j C_{\varphi_j}$ is compact on $\mathcal H$ if and only if each $C_{\varphi_j}$ is compact.
\end{theorem}
\begin{proof}
Arguing as at the beginning of the proof of Theorem \ref{thm:lcpolynomial},
we can assume that $\varphi_1$ has a positive characteristic,
and $\sum_{j=1}^N \lambda_j C_{\varphi_j}$ is compact
whereas no $C_{\varphi_j}$ is compact. Then it will be shown that $\lambda_1=0$, which leads to a contradiction with the hypotheses.
We can also assume that $\varphi_1(0)=0$. We write each $\varphi_j$ as
 $$\varphi_j(s)=c_{0,j}s+c_{1,j}+\sum_{l=1}^{d^{(j)}} c_{q_l^{(j)}}(q_l^{(j)})^{-s}=c_{0,j}s+\psi_j(s)$$
 where the integers $(q_l^{(j)})_{j=1,\dots,d^{(j)}}$ are multiplicatively independent and the complex numbers $c_{q_l^{(j)}}$ are non-zero.
 Since $\varphi_j$ is supposed to have unrestricted range and the $(q_l^{(j)})$ are multiplicatively independent, we know that $\A (c_{1,j})=\sum_{l=1}^{d^{(j)}} |c_{q_l^{(j)}}|$ if $c_{0,j}\neq 0$,
 and $\A (c_{1,j})=\frac 12+\sum_{l=1}^{d^{(j)}} |c_{q_l^{(j)}}|$ otherwise.
 To simplify the notation, we set $c_0=c_{0,1}$,
 $d=d^{(1)}$, $q_l=q_l^{(1)}$, $c_{q_l}=c_{q_l^{(1)}}$ and $\mathcal Q=\{q_1,\dots,q_d\}$. We set $\mathcal Q_{\mathbb N}=\{q\in\mathbb N:\ \exists \alpha\in\mathbb Q^d,\ q=q_1^{\alpha_1}\cdots q_d^{\alpha_d}\}$
 and observe that if $q\notin \mathcal Q_{\mathbb N}$, then the $(d+1)$ integers $q_1,\dots,q_d,q$ are multiplicatively independent.

 Inspired by the proof of Theorem \ref{thm:lcpolynomial}, we set
 \begin{align*}
  I&=\left\{j:\ c_{0,j}=c_0,\ \varphi_j(0)=0,\ q_l^{(j)}\in\mathcal Q_\mathbb N,\ l=1,\dots,d^{(j)}\right\}\\
  J_1&=\left\{j:\ c_{0,j}=c_0,\ \varphi_j(0)\neq 0,\ q_l^{(j)}\in\mathcal Q_\mathbb N,\ l=1,\dots,d^{(j)}\right\}\\
  J_2&=\left\{j:\ c_{0,j}=c_0,\ \exists l\in\{1,\dots,d^{(j)}\},\ q_l^{(j)}\notin \mathcal Q_\mathbb N\right\}\\
  J_3&=\left\{j:\ c_{0,j}\neq c_0\textrm{ and }c_{0,j}>0\right\}=:\{j_1,\dots,j_p\}\\
  J_4&=\left\{j:\ c_{0,j}=0\right\}.
 \end{align*}

 We start by working with indices in $I$. Therefore, let $j\in I$. The map $\varphi_j$ may be written
 $$\varphi_j(s)=c_0 s+\sum_{\alpha\in\Gamma_j} c_\alpha^{(j)}\big((q^\alpha)^{-s}-1\big)$$
 with $\Gamma_j\subseteq\mathbb Q^d\backslash\{(0,\dots,0)\}$ and $c_{\alpha}^{(j)}<0$ for all $\alpha\in\Gamma_j$. In particular
 $$\mathcal B_{\mathcal Q}\psi_j(z)=\sum_{\alpha\in\Gamma_j}c_\alpha^{(j)}\left(\prod_{l=1}^d z_l^{\alpha_l}-1\right).$$
  For $j\in I$, let us introduce
 \begin{align*}
  L_j(x_1,\dots,x_d)&=\sum_{l=1}^d \sum_{\alpha\in\Gamma_j}c_\alpha^{(j)}\alpha_l\log(p_l)x_l\\
  Q_j(x_1,\dots,x_d)&=\frac 12\sum_{1\leq l, m\leq d}\sum_{\alpha\in \Gamma_j}c_\alpha^{(j)}\alpha_l\alpha_m \log(p_l)\log(p_m)x_lx_m.
 \end{align*}
 Let also $\mathcal C=\{j\in I:\ L_j=L_1\}$ and pick $j\in I$, $j\neq 1$. If $Q_j=Q_1$, then
 \begin{align}\label{eq:lclinear1}
  \textrm{for all }1\leq l\neq m\leq d,\ \sum_{\alpha\in\Gamma_j}c_\alpha^{(j)}\alpha_l\alpha_m&=0\\
  \label{eq:lclinear2}
  \textrm{for all }1\leq l\leq d,\ \sum_{\alpha\in\Gamma_j}c_\alpha^{(j)}\alpha_l^2&=c_{q_l}.
 \end{align}
 Since each $c_\alpha^{(j)}$ is negative, \eqref{eq:lclinear1} tells us that for any $\alpha\in\Gamma_j$, at most
 one of its component is non-zero. Moreover, since the integers $q^\alpha$, $\alpha\in\Gamma_j$, are multiplicatively
 independent, if $l\in\{1,\dots,d\}$, at most one $\alpha\in\Gamma_j$ can have its non-zero component equal to its $l$-th component.
 Therefore, the sum appearing in \eqref{eq:lclinear2} reduces to at most one element. Since $c_{q_l}\neq 0$, the sum cannot be empty
 and we have shown that $\textrm{card}(\Gamma_j)=d$, that $\Gamma_j$ may be written $\{\alpha(1),\dots,\alpha(d)\}$ with $\alpha(l)=(0,\dots,0,\alpha_l(l),0,\dots,0)$, the non-zero term $\alpha_l(l)$ appearing at the $l-$th position.
 Now, if we compare \eqref{eq:lclinear2} with $\alpha_l(l)d_{\alpha(l)}^{(j)}=\sum_{\alpha\in\Gamma_j}\alpha_l c_\alpha^{(j)}=c_{q_l}$
 (which comes from $L_j=L_1$), we get that each $\alpha_l(l)$ has to be equal to $1$.
 Finally, we have shown that $\Gamma_j=\{(1,0,\dots,0),(0,1,0,\dots,0),\dots,(0,\dots,0,1)\}$ which in turn yields $\varphi_j=\varphi_1$,
 a contradiction.

From this, we deduce using an argument similar to that of the proof of Theorem \ref{thm:lcpolynomial} the existence
of $\beta\in\RR^d$ such that, for all $j\in I$, $j\neq 1$, either $L_j(\beta)\neq L_1(\beta)$ or $Q_j(\beta)\neq Q_1(\beta)$.
We now set, for $1\leq l\leq d$, $k\geq 1$ and $M>0$,
$$z_l(k)=p_l^{-\frac Mk+\frac{i\beta_l}{\sqrt k}}$$
so that
$$z_l(k)^{\alpha_l}=1+i\frac{\alpha_l \log(p_l)\beta_l}{\sqrt k}-\frac{M\alpha_l \log(p_l)+\frac 12\alpha_l^2\log^2(p_l)\beta_l^2}{k}+o\left(\frac 1k\right).$$
This yields for any $j\in I$,
\begin{align*}
 \mathcal B_{\mathcal Q}\psi_j(z(k))&=\frac{i}{\sqrt k}\sum_{\alpha\in\Gamma_j}c_\alpha^{(j)}\sum_{l=1}^d \alpha_l\log(p_l)\beta_l-\frac 1k
 \sum_{\alpha\in\Gamma_j}c_\alpha^{(j)}\Bigg(\sum_{l=1}^d M\alpha_l \log(p_l)\\
 &\quad+\frac 12\sum_{l=1}^d \alpha_l^2 \log^2(p_l)\beta_l^2+\sum_{1\leq l<m\leq d}c_\alpha^{(j)}
 \alpha_l\alpha_m \log(p_l)\log(p_m)\beta_l\beta_m\Bigg)\\
 &\quad\quad+o\left(\frac 1k\right)\\
 &=\frac{i L_j(\beta)}{\sqrt k}-\frac{L_j(\mathbf e)+Q_j(\beta)}{k}+o\left(\frac 1k\right).
\end{align*}

Suppose now that $j\in J_1$ which means that $\mathcal B_{\mathcal Q}\psi_j(\mathbf e)\neq 0$.
Hence, $\mathcal B_{\mathcal Q}\psi_j(z(k))$ tends to some $a_j\neq 0$. To handle $J_2$,
we first introduce integers $q_{d+1},\dots,q_e$ such that $q_1,\dots,q_d,q_{d+1},\dots, q_e$ is a maximal sequence of multiplicatively
independent integers extracted from $\{q_l^{(j)}:\ 1\leq j\leq N,\ 1\leq l\leq d^{(j)}\}$.
For $j\in J_2$, each $q_l^{(j)}$ may be written $q_1^{\alpha_1}\cdots q_d^{\alpha_d}q_{d+1}^{\alpha_{d+1}}\cdots q_e^{\alpha_e}$
with $\alpha_i\in\mathbb Q$. Furthermore, there exists some $l(j)$ such that the family $(\alpha_{d+1},\dots,\alpha_e)$
is not identically zero (otherwise $j$ would belong to $I\cup J_1$). Denote this family by $(\alpha_{d+1}(j),\dots,\alpha_e(j))$. Then
\[ E_j:=\left\{(\theta_{d+1},\dots,\theta_e)\in\mathbb R^{e-d}:\ e^{i(\theta_{d+1}\alpha_{d+1}(j)+\cdots+\theta_e\alpha_e(j)}c_{q_j^{(l)}}=|c_{q_l^{(j)}}|\right\} \]
has Lebesgue measure zero.
Therefore we can pick $(\theta_{d+1},\dots,\theta_e)\in\mathbb R^{e-d}\backslash \bigcup_{j\in J_2}E_j$. Denoting $\mathcal Q'=\{q_{d+1},\dots,q_e\}$, since
$\A (c_{1,j})=\sum_{l=1}^{d^{(j)}} |c_{q_l^{(j)}}|$, this implies that
$$\mathcal B_{\mathcal Q\cup\mathcal Q'}\psi_j(1,\dots,1,e^{i\theta_{d+1}},\dots,e^{i\theta_e})=:a_j\neq 0.$$
To proceed with the conclusion,
we finally fix some {positive integers} $q_{e+1},\dots,q_{e+p}$ such that $q_1,\dots,q_{e+p}$ are multiplicatively independent (recall that $p=\textrm{card}(J_3)$).
We then choose, by Kronecker's theorem, a sequence of real numbers $(t_k)$ such that
\begin{enumerate}[(a)]
 \item for $1\leq l\leq d$, $q_l^{-it_k}$ is very close to $q_l^{-\frac{i\beta_l}{\sqrt k}}$, so close that, setting $w_k=\frac{-M}k+it_k$,
 $$\mathcal B_{\mathcal Q}\psi_j(q_1^{-w_k},\dots,q_d^{-w_k})=\mathcal B_{\mathcal Q}\psi_j(z(k))+o\left(\frac 1k\right)$$
 for all $j\in I\cup J_1$.
 \item for $d+1\leq l\leq e$, $q_l^{-it_k}$ tends to $e^{i\theta_l}$, so that
 $${\mathcal B_{\mathcal Q\cup\mathcal Q'}\psi_j(q_1^{-w_k},\dots,q_e^{-w_k}) } =a_j+o\left(\frac 1k\right)$$
 for all $j\in J_2$.
 \item for all $j\in J_3$, writing $j=j_r$ with $1\leq r\leq p$,
 \[ q_{e+r}^{-it_k(c_{0,j_r}-c_0)}q_{e+r}^{-\mathcal B_{\mathcal Q\cup\mathcal Q'}\psi_j(1,\dots,1,e^{i\theta_{d+1}},\dots,e^{i\theta_e})}\textrm{ tends to some }\omega(j_r)\neq 1. \]
\end{enumerate}
Let $d'$ be such that $p^+(q_l)\leq p_{d'}$ for all $l=1,\dots,e+r$ and $p^+(q_l^{(j)})\leq p_{d'}$ for all $j\in J_4$ and all $1\leq l\leq d^{(j)}$.
We first observe that, for $j\in\mathcal C$ and $j'\notin \mathcal C$,
$$\langle C_{\varphi_{j'}}^*(k_{d',w_k}),C_{\varphi_{j}}^*(k_{d',w_k})\rangle\xrightarrow{k\to+\infty}0. $$
When $j'\notin J_4$, this follows from the choice of the sequence $t_k$, and the fact that for $l\in\{1,\dots,N\}$, $l\notin J_4$,
\begin{equation*}
 \varphi_l(w_k)=c_{0,l}w_k+\mathcal B_{\mathcal Q\cup\mathcal Q'}\psi_l(q_1^{-w_k},\dots,q_e^{-w_k}).
\end{equation*}
When $j'\in J_4$, this has already been observed during the proof of Theorem  \ref{thm:essentialnormcomposition}.

We now argue exactly as in the proof of Theorem \ref{thm:essentialnormcomposition}.
to conclude that, for all $M>0$,
\begin{equation} \label{eq:lclinear3}
  \sum_{j,j'\in\mathcal C} \frac{\lambda_j\overline{\lambda_{j'}}}{\big(2M(c_0-L_1(\mathbf e))-(\overline{Q_{j'}(\beta)}+Q_j(\beta))\big)^{d'}}=0.
\end{equation}
Finally we consider the following equivalence relation
on $\mathcal C$: $j\mathcal R j'\iff Q_j(\beta)=Q_{j'}(\beta)$. Let $\mathcal C_1,\dots,\mathcal C_\kappa$ be the equivalent classes for this relation, with $\mathcal C_1=\{1\}$,
let $\nu_k=\sum_{j\in\mathcal C_k}\lambda_j$ and let $\mu_k=Q_j(\beta)$ for some (all) $j\in\mathcal C_k$, $1\leq k\leq\kappa$. Then
\eqref{eq:lclinear3} may be rewritten
\[
 \sum_{k,k'=1}^{\kappa} \frac{{\nu_{k}}\overline{\nu_{k'}}}{\big(2M(c_0-L_1(\mathbf e))-(\overline{\mu_{k'}}+\mu_k)\big)^{d'}}=0. \]
Lemma \ref{lem:vanishingquadratic} tells us that, for all $k=1,\dots,\kappa$, $\nu_k=0$. In particular, $\lambda_1=\nu_1=0$, a contradiction!
\end{proof}

To tackle the case where all symbols have zero characteristic, we need one elementary lemma:

\begin{lemma}\label{lem:powersingle}
 Let $\mathcal Q$ be a finite set of positive integers. Assume that for any $n,m\in\mathcal Q$, there exist $\alpha,\beta\in\mathbb N$
 such that $n^{\alpha}=m^{\beta}$. Then there exists $q\in\mathbb N$ such that, for any $n\in\mathcal Q$, there exists $\alpha\in\mathbb N$
 with $n=q^\alpha$.
\end{lemma}
\begin{proof}
 We proceed by induction on the cardinal $N$ of $\mathcal Q$, the result being trivial for $N=1$ and $N=2$. Let us assume that it has been proved until $N$
 and let $\mathcal Q=\{q_1,\dots,q_{N+1}\}$ be a set with $N+1$ elements. We apply the inductive step to find some $q'$ working with $\mathcal Q'=\{q_1,\dots,q_N\}$ and then we apply the result for $N=2$
 to $\{q',q_{N+1}\}$ to conclude.
\end{proof}

\begin{theorem}\label{thm:lclinearzero}
Let $(\varphi_j)_{j=1,\dots,N}\subseteq\mathcal G$ be a finite sequence of distinct linear symbols with zero characteristic.
Let $(\lambda_j)_{j=1,\dots,N}$ be a finite sequence of non-zero complex numbers. Then $\sum_{j=1}^N \lambda_j C_{\varphi_j}$ is compact on $\mathcal H$ if and only if each $C_{\varphi_j}$ is compact.
\end{theorem}
\begin{proof}
 Arguing as above, we start with $N$ linear symbols $\varphi_1,\dots,\varphi_N$ having zero characteristic and with $N$ non-zero complex numbers $\lambda_1,\dots,\lambda_N$
 such that $\sum_{j=1}^N \lambda_j C_{\varphi_j}$ is compact whereas no $C_{\varphi_j}$ is compact. In particular, one may write $\varphi_j(s)=c_j+d_j q_j^{-s}$ {where $|d_j|=\A(c_j)-\frac 12$.}
 Without loss of generality, we shall assume that $\varphi_1(0)=\frac 12$.

 As before, we split $\{1,\dots,N\}$ into disjoint subsets:
\begin{align*}
 I&=\left\{j:\ \varphi_j(0)=\varphi_1(0)\textrm{ and }q_j^\alpha=q_1^{\beta}\textrm{ for some }\alpha,\beta\in\mathbb N\right\}\\
 J_1&=\left\{j:\ \varphi_j(0)\neq \varphi_1(0)\textrm{ and }q_j^\alpha=q_1^{\beta}\textrm{ for some }\alpha,\beta\in\mathbb N\right\}\\
 J_2&=\left\{j:\ q_j^{\alpha}\neq q_1^{\beta}\textrm{ for all }\alpha,\beta\in\mathbb N\right\}.
\end{align*}
Let $q$ be given by Lemma \ref{lem:powersingle} with $\mathcal Q=\{q_j:\ j\in I\cup J_1\}$, for $j\in I$, let $\alpha_j\in\mathbb N$ be such that $q_j=q^{\alpha_j}$
and, for $M>0$, consider $w_k=\frac Mk+\frac{i}{\sqrt k}$. When $j\in I$,
\begin{align*}
 \varphi_j(w_k)&=c_j+d_j \big(q^{-\alpha_j w_k}-1\big)+d_j\\
 &=\frac 12-\frac{i\alpha_j d_j \log q}{\sqrt k}-\frac{2M\alpha_j d_j \log q+\alpha_j^2 d_j\log^2 q}{2k}+o\left(\frac 1k\right)
\end{align*}
and let us recall that $d_j<0$. Moreover, when $j\in J_2$, applying Lemma \ref{lem:otherk},
\[ C_{\varphi_j}^* (k_{w_k}^{(q)})=\frac{K_{c_{j}}}{\|K_{w_k}^{(q)}\|}\xrightarrow{k\to+\infty}0. \]
Therefore, one gets
\begin{align*}
 0&\geq \left\|\sum_{j=1}^N \lambda_j C_{\varphi_j}\right\|_e^2\\
 &\geq \limsup_{k\to+\infty} \left\|\sum_{j=1}^N \overline{\lambda_j} C_{\varphi_j}^*(k_{w_k}^{(q)})\right\|^2\\
 &\geq \limsup_{k\to+\infty} \left\|\sum_{j\in I\cup J_1}\overline{\lambda_j} C_{\varphi_j}^*(k_{w_k}^{(q)})\right\|^2\\
 &\geq \limsup_{k\to+\infty} \sum_{j,j'\in I\cup J_1} \frac{\lambda_j \overline{\lambda_{j'}} K_{\varphi_{j'}(w_k)}(\varphi_j(w_k))}{\|K_{w_k}^{(q)}\|^2}\\
 &\geq \limsup_{k\to+\infty} \sum_{j,j'\in I\cup J_1} \lambda_{j}\overline{\lambda_{j'}}\frac{\zeta\big(\overline{\varphi_{j'}(w_k)}+\varphi_j(w_k)\big)}{\|K_{w_k}^{(q)}\|^2}.
\end{align*}
Now, if $j\in I$ and $j'\in J_1$, $\overline{\varphi_{j'}(w_k)}+\varphi_j(w_k)$ tends to some complex number in $\overline{\mathbb C_1}$ which is not equal to $1$. Hence,
\[  \frac{\zeta\big(\overline{\varphi_{j'}(w_k)}+\varphi_j(w_k)\big)}{\|K_{w_k}^{(q)}\|^2} \xrightarrow{k\to+\infty}0 \]
so that, arguing as in the proof of Theorem \ref{thm:lcpolynomial},
\[ \limsup_{k\to+\infty}\sum_{j,j'\in I}\lambda_j \overline{\lambda_{j'}} \frac{\zeta\big(\overline{\varphi_{j'}(w_k)}+\varphi_j(w_k)\big)}{\|K_{w_k}^{(q)}\|^2}=0. \]
Let finally $\mathcal C=\{j\in I:\ \alpha_j d_j=\alpha_1 d_1\}$. Using $\zeta(s)\sim_{s\to 1}\frac1{s-1}$ and $\|K_w^{(q)}\|^2\sim_{\A w\to 0}\frac{C_q}{\A w}$, we finally
get
\[ \sum_{j,j'\in\mathcal C} \frac{ \lambda_j\overline{\lambda_{j'}}}{4M\alpha_1d_1\log q+\alpha_j^2 d_j\log^2 q+\alpha_{j'}^2 d_{j'}\log^2q}=0. \]
Since for all $j\neq j'\in\mathcal C$, it is clear that $\alpha_j^2 d_j\neq \alpha_{j'}^2d_{j'}$ (otherwise $\alpha_j=\alpha_{j'}$ and $d_{j'}=d_j$)
we conclude by Lemma \ref{lem:vanishingquadratic} that all $\lambda_j$, $j\in\mathcal C$,
are equal to zero, a contradiction.
\end{proof}


\section{Weigthed composition operators, connected components and compact differences}

\label{sec:sufficient}

\subsection{Introduction}
Our aim in this section is to give sufficient conditions to prove
that two composition operators {$\cphiz$ and $\cphiu$} belong to the same connected component of $\mathcal C(\mathcal H)$,
or that their difference is compact.
We shall see that this could be deduced from results on weighted composition operators (see \cite{MT} where this idea seems to appear for the first time). We start with
$\varphi_0,\ \varphi_1\in\mathcal G$ with the same characteristic $c_0$. The simplest way to draw an
arc between $\cphiz$ and $\cphiu$ is to consider $\lambda\in[0,1]\mapsto C_{\varphi_\lambda}$,
where $\varphi_\lambda=(1-\lambda)\varphi_0+\lambda\varphi_1$. To prove that $\lambda\mapsto C_{\varphi_\lambda}$ is continuous, we start as in the proof of Theorem \ref{thm:connectedcompact}: let $\mu$ be a probability measure on $\RR$, $f\in\mathcal H$ with $\|f\|\leq 1$ and let us write
\[ \left\|C_{\varphi_\lambda}(f)-C_{\varphi_{\lambda'}}(f)\right\|^2=\int_{\ttinf}\int_{\RR}|f_{\chi^{c_0}}\circ(\varphi_\lambda)_\chi(it)-f_{\chi^{c_0}}\circ(\varphi_{\lambda'})_\chi(it)|^2 d\mu(t)dm(\chi). \]
At this point we do not use the mean value theorem but instead we write
$$\int_{\lambda}^{\lambda'} f'_{\chi^{c_0}}((\varphi_r)_\chi(it))dr=\frac 1{(\varphi_1)_\chi(it)-(\varphi_0)_\chi(it)}
\left(f_{\chi^{c_0}}\circ (\varphi_\lambda)_\chi(it)
-f_{\chi^{c_0}}\circ (\varphi_{\lambda'})_\chi(it)\right)$$
so that, using Jensen's inequality,
\begin{eqnarray}
\label{eq:wco1}
\left\|C_{\varphi_\lambda}(f)-C_{\varphi_{\lambda'}}(f)\right\|^2 &\leq&\displaystyle |\lambda'-\lambda|^2
\int_\lambda^{\lambda'}\!\int_{\ttinf}\!\int_{0}^1\left|(\varphi_1)_\chi(it)-(\varphi_0)_\chi(it)\right|^2\\[4mm]
&&\nonumber \quad\quad\quad\quad\quad\quad\quad\times
 |f'_{\chi^{c_0}}((\varphi_r)_\chi(it))|^2d\mu(t) dm(\chi)dr.
\end{eqnarray}
This inequality leads us to the following lemma.

\begin{lemma}\label{lem:wco}
Let $\varphi_0,\ \varphi_1\in \mathcal G$ with $\cha(\varphi_0)=\cha(\varphi_1)$ and let $\varphi_{\lambda}=(1-\lambda)\varphi_0+\lambda\varphi_1$, $\lambda\in[0,1]$.
\begin{enumerate}[(i)]
\item Suppose that for all $\lambda\in[0,1]$, $f\mapsto (\varphi_0-\varphi_1)C_{\varphi_{\lambda}}(f')$ acts boundedly on $\mathcal H$ with uniformly bounded norms. Then there is a constant $A>0$ such that, for all $\lambda,\lambda'\in[0,1]$,
\[ \left\|C_{\varphi_\lambda}-C_{\varphi_{\lambda'}}\right\|\leq A |\lambda-\lambda'|. \]
\item Suppose moreover that for each $\lambda\in[0,1]$, $f\mapsto (\varphi_0-\varphi_1)C_{\varphi_\lambda}(f')$ is a compact operator on $\mathcal H$. Then $\cphiz-\cphiu$ is compact.
\end{enumerate}
\end{lemma}
\begin{proof}
(i) is an immediate consequence of \eqref{eq:wco1} whereas (ii) follows from an application of Lebesgue's dominated convergence theorem.
\end{proof}

Hence we are reduced to study the boundedness and the compactness of a weighted composition and derivation operator. By working with a Bergman space instead of a Hardy space as initial space, we will be reduced to study the same problems for weighted composition operators acting between different spaces. At this point, we will divide the proof into two cases, following the value of $c_0$. For $c_0>0$ we will stay on the halfplane (and on the hidden polydisc) and use Nevanlinna counting functions. For $c_0=0$ we will work on the disk and use Carleson measures.

\subsection{The case of positive characteristic}
\subsubsection{Boundedness and compactness of weighted composition operators}

We shall need the following lemma which explains how to compute the norm of an element of $\mathcal H$ using its derivative.

\begin{lemma}\label{lem:normH}
Let $\mu$ be a probability measure on $\RR$ and $f=\sum_n a_nn^{-s}$.
Then the following assertions are equivalent:
\begin{enumerate}[(i)]
\item $f\in\mathcal H$;
\item for almost all $\chi\in\ttinf$, $f_\chi$ extends to $\CC_+$ and
$$\tripleint|f_\chi'(\sigma+it)|^2\sigma d\sigma d\mu(t)dm(\chi)<+\infty.$$
\end{enumerate}
Moreover, provided $f\in\mathcal H$, then
\begin{align}
\|f\|^2&\approx |a_1|^2+\tripleint|f_\chi'(\sigma+it)|^2\sigma d\sigma d\mu(t)dm(\chi). \label{eq:normH1}
\end{align}
\end{lemma}
\begin{proof}
We just need to prove \eqref{eq:normH1} for a Dirichlet polynomial. Now by \eqref{eq:normH0} and Fubini's theorem,
\begin{align*}
\tripleint|f_\chi'(\sigma+it)|^2\sigma d\sigma d\mu(t)dm(\chi) &= \sum_{n\geq 2} |a_n|^2 \log^2(n) \int_0^1 \sigma n^{-2\sigma}d\sigma\\
&\approx \sum_{n\geq 2}|a_n|^2.
\end{align*}
\end{proof}

This leads us to introduce the following Banach space of Dirichlet series: $\mathcal A$ is the completion of the space of Dirichlet polynomials for the norm
\begin{eqnarray}
\nonumber \|f\|_{\mathcal A}^2&:=&\int_0^1 \|f(\cdot+\sigma)\|^2_{\mathcal H}\sigma d\sigma\\
\label{eq:wco2}&=& \int_{\ttinf}\int_{\mathbb R}\int_0^1 |f_\chi(\sigma+it)|^2 \sigma d\sigma d\mu(t) dm(\chi).
\end{eqnarray}

This space, a particular example of the Bergman spaces introduced in \cite{BL15}, is a space of Dirichlet series defined on $\CC_{1/2}$. Moreover,
if $f\in\mathcal A$, for almost all $\chi\in\ttinf$, $f_\chi$ converges on $\CC_+$ and the norm in $\mathcal A$ can be computed using \eqref{eq:wco2}.

Lemma \ref{lem:normH} says that $f\in\mathcal H$ if and only if $f'\in\mathcal A$ and that up
to the constant term, the norms are comparable. Moreover, it is easy to check that if a sequence $(f_n)$ converges weakly to $0$
in $\mathcal H$, then $(f_n')$ converges weakly to $0$ in $\mathcal A$. Therefore we are led to the study of the boundedness and of the compacness of weighted composition operators from $\mathcal A$ to $\mathcal H$. This is the content of the next theorem, of independent interest.

\begin{theorem}\label{prop:wcopositive}
Let $\varphi\in\mathcal G$ with $\cha(\varphi)>0$ and write it $\varphi=c_0s+\psi$.
Let $u$ be a Dirichlet series convergent in $\CC_+$. Assume that there exists $C>0$ such that
\begin{itemize}
\item $|u'|\leq C,\  |\psi| \leq C$;
\item $|u|\leq C \A \varphi$.
\end{itemize}
Then $\|uC_\varphi\|_{\mathcal A\to\mathcal H}\lesssim_{c_0,C} 1$.
\end{theorem}

The proof will need the following analogue of Lemma \ref{lem:normH} for $\mathcal A$, which has a completely similar proof: for all $f\in\mathcal A$, for all $c_0\geq 1$,
\[ \|f\|_{\mathcal A}^2 \approx |a_1|^2+\int_{\ttinf}\!\int_{\RR}\!\int_0^{+\infty}|f_{\chi^{c_0}}'(\sigma+it)|^2\sigma^3 d\sigma d\mu(t)dm(\chi).\]

\begin{proof}[Proof of Theorem \ref{prop:wcopositive}]
Let $f\in\mathcal A$. We choose for $\mu$ the probability measure $d\mu = \mathbf 1_{[0,1]}dt$.
We intend to show that $\|u C_\varphi(f)\|_{\mathcal H}\lesssim_{c_0,C}\|f\|_{\mathcal A}.$
By Lemma \ref{lem:normH}, it suffices to show that $\|(uC_\varphi(f))'\|_{\mathcal A}\lesssim_{c_0,C}\|f\|_{\mathcal A}.$ Now,
\[ (uC_\varphi f)'=u' f\circ \varphi+u\varphi'  f'\circ\varphi. \]

Therefore we have to study the following integrals:
$$I_1=\int_{\ttinf}\!\int_0^1\!\int_0^1 |u_\chi'(\sigma+it)|^2 |f_{\chi^{c_0}} \circ \varphi_\chi(\sigma+it)|^2 \sigma d\sigma dt dm(\chi)$$
$$I_2=\int_{\ttinf}\!\int_0^1\!\int_0^1 |u_\chi(\sigma+it)|^2 |\varphi_\chi'(\sigma+it)|^2|f'_{\chi^{c_0}}\circ \varphi_\chi(\sigma+it)|^2 \sigma d\sigma dt dm(\chi).$$
We first handle $I_1$ and observe that, $u'_\chi$ being a vertical translate of
$u'$,  for all $\chi\in\ttinf$, for all $s\in\CC_+$, $|u'_\chi(s)|\leq C$.
On the other hand, we appeal to \cite{bai} to ensure that
\[ I_1\leq C \|f\circ\varphi\|_{\mathcal A}\leq C \|f\|_{\mathcal A}. \]
Let us turn to $I_2$. Using again vertical translates, we know that, for all $s\in\CC_+$,
$|u_\chi(s)|\leq C \A \big( \varphi_\chi(s)\big).$
Since we also know that $\A(s)\leq \A \big( \varphi_\chi(s)\big)/c_0$,
we obtain
$$I_2\lesssim_{c_0,C} \int_{\ttinf}\!\int_0^1\!\int_0^1 |\varphi_\chi'(\sigma+it)|^2|f'_{\chi^{c_0}}\circ \varphi_\chi(\sigma+it)|^2   \A^3 \big( \varphi_\chi(\sigma+it)\big) d\sigma dt dm(\chi).$$
We then do, for a fixed $\chi\in\ttinf$, the nonunivalent change of variables $w=\varphi_\chi(s)$.
We denote by $\mathcal N_{\varphi_\chi,3}$ the appropriated Nevanlinna counting function, namely
$$\mathcal N_{\varphi_\chi,3}(w)=\sum_{\varphi_\chi(w)=s}\A^3 (s).$$
Observing that our assumptions ensure that for all $\chi\in\ttinf$, $\varphi_\chi([0,1]^2)\subseteq[0,c_0+C]\times [-c_0-C,c_0+C]$, we then get
$$I_2\lesssim_{c_0,C} \int_{\ttinf}\!\int_{-c_0-C}^{c_0+C}\!\int_0^{c_0+C}|f'_{\chi^{c_0}}(\sigma+it)|^2  \mathcal N_{\varphi_\chi,3}(\sigma+it) d\sigma dt dm(\chi).$$
Now, since
\begin{align*}
N_{\varphi_\chi,3}(w)& \leq \left( \sum_{\varphi_\chi(w)=s}\A (s) \right)^3\\
&\leq c_0^3 \A^3(w)
\end{align*}
by \cite[Proposition 3]{b2}, we can finally conclude that
$$I_2\lesssim_{c_0,C}  \int_{\ttinf}\!\int_{-C-c_0}^{C+c_0}\!\int_0^{c_0+C}|f'_{\chi^{c_0}}(\sigma+it)|^2  \sigma^3 d\sigma dt dm(\chi)\lesssim_{c_0,C} \|f\|_{\mathcal A}^2.$$
\end{proof}

Let us now turn to compactness.

\begin{theorem}\label{prop:wcocompactnesspositive}
Let $\varphi\in\mathcal G$ with $\cha(\varphi)>0$ and write it $\varphi=c_0s+\psi$.
Let $u$ be a Dirichlet series convergent in $\CC_+$. Assume that there exists $C>0$ such that
\begin{itemize}
\item $|u'|\leq C,\  |\psi| \leq C$;
\item $|u|\leq C \A \varphi$;
\item $|u'|\to 0$ and $|u|=o\big(\A(\varphi)\big)$ as $\A(\varphi)\to 0$.
\end{itemize}
Then $uC_\varphi$ is a compact operator from $\mathcal A$ to $\mathcal H$.
\end{theorem}
\begin{proof}
Let $(f_{n})$ be a sequence in $\mathcal A$ which converges weakly to $0$ and satisfies $\|f_{n}\|_{\mathcal A}\leq1$.
Arguing as in the proof of Theorem \ref{prop:wcopositive}, we have to show that $I_{n,1}$ and $I_{n,2}$ go to zero where
$$I_{n,1}=\int_{\ttinf}\!\int_0^1\!\int_0^1 |u_\chi'(\sigma+it)|^2 |f_{n,\chi^{c_0}} \circ \varphi_\chi(\sigma+it)|^2 \sigma d\sigma dt dm(\chi),$$
$$I_{n,2}=\int_{\ttinf}\!\int_0^1\!\int_0^1 |u_\chi(\sigma+it)|^2 |\varphi_\chi'(\sigma+it)|^2|f'_{n,\chi^{c_0}}\circ \varphi_\chi(\sigma+it)|^2 \sigma d\sigma dt dm(\chi).$$

We first handle $I_{n,1}$.
Let $\veps>0$ and $\theta\in(0,1)$ be such that, for all $\chi\in\ttinf$ and for all $s\in\CC_+$ with $\A \varphi_\chi(s)\in(0,\theta)$,
$|u'_\chi(s)|\leq\veps.$
Then, setting $\Omega_{\chi}=\left\{(\sigma,t)\in [0,1]^2:\ \A (\varphi_\chi(s))\in (0,\theta)\right\}$,
\begin{align*}
 &\int_{\ttinf}\!\int\!\!\int_{\Omega_\chi}|u_\chi'(\sigma+it)|^2 |f_{n,\chi^{c_0}} \circ \varphi_\chi(\sigma+it)|^2 \sigma d\sigma dt dm(\chi)\\
 &\quad\quad\leq \veps^2\int_{\ttinf}\!\int_0^1\!\int_0^1 |f_{n,\chi^{c_0}} \circ \varphi_\chi(\sigma+it)|^2 \sigma d\sigma dt dm(\chi)\\
 &\quad\quad\leq \veps^2.
\end{align*}
On the other hand, for all $\chi\in\ttinf$, for all $(\sigma,t)\in [0,1]^2\backslash \Omega_\chi$, setting $s=\sigma+it$, $\A(\varphi_\chi(s))\geq \theta$
and $|\B(\varphi_\chi(s))|\leq C+c_0$. Therefore, there exists $C_\theta>0$ such that, for all $\chi\in\ttinf$ and all such $(\sigma,t)$,
\[ |f_{n,\chi^{c_0}} \circ \varphi_\chi(\sigma+it)|^2 \leq C_\theta \|f_{n,\chi^{c_0}}(\cdot+\theta/2)\|^2_{H^2_i(\CC_+)}, \]
where $H^2_i(\CC_+)$ is the invariant Hardy space on the right halfplane, namely the set of analytic functions $g:\CC_+\to\CC_+$ satisfying
\[ \|g\|_{H^2_i(\CC_+)}^2:=\frac 1\pi \int_{-\infty}^{+\infty}|g(it)|^2 \frac{dt}{1+t^2}<+\infty. \]
This yields
\begin{align*}
 &\int_{\ttinf}\!\int\!\!\int_{[0,1]^2\backslash \Omega_\chi}|u_\chi'(\sigma+it)|^2 |f_{n,\chi^{c_0}} \circ \varphi_\chi(\sigma+it)|^2 \sigma d\sigma dt dm(\chi)\\
 &\quad\quad\lesssim_\theta \int_{\ttinf}\!\int_0^1\!\int_0^1\!\int_{\RR}\left|f_{n,\chi^{c_0}}\left(iu+\frac{\theta}2\right)\right|^2\frac{du}{1+u^2}\sigma d\sigma dtdm(\chi)\\
 &\quad\quad\lesssim_\theta \sum_{k\geq 1}|a_{n,k}|^2 k^{-\theta},
\end{align*}
where we have written each $f_n$ as $f_n=\sum_k a_{n,k}k^{-s}$ and we have used
\eqref{eq:normH0} for $d\mu(u)=\frac{du}{1+u^2}$. Since $(f_n)$ tends weakly to zero and $\theta>0$, it is easy to prove that $\sum_k |a_{n,k}|^2 k^{-\theta}$
goes to zero as $n$ tends to $+\infty$. Hence we have shown that for all $\veps>0$, $\limsup_{n\to+\infty}I_{n,1}\leq\veps$ which implies that $\lim_{n\to+\infty}I_{n,1}=0$.

Let us turn to $I_{n,2}$. Again for arbitrary $\veps>0$ there exists $\theta\in(0,1)$
such that $\A(\varphi_\chi(s))<\theta$ implies
$$|u_\chi(s)|\leq \veps \A \big(\varphi_\chi(s)\big).$$
Keeping the notation $\Omega_{\chi}=\left\{(\sigma,t)\in [0,1]^2:\ \A (\varphi_\chi(s))\in (0,\theta)\right\}$
and using $\A(s)\leq \A \big( \varphi_\chi(s)\big)/c_0$,
We obtain
\begin{align*}
 & \int_{\ttinf}\!\int\!\!\int_{\Omega_\chi} |u_\chi(\sigma+it)|^2 |\varphi_\chi'(\sigma+it)|^2|f'_{n,\chi^{c_0}}\circ \varphi_\chi(\sigma+it)|^2 \sigma d\sigma dt dm(\chi) \\
 & \leq    \frac{1}{c_{0}}\veps^{2}\int_{\ttinf}\!\int_0^1\!\int_0^1  |\varphi_\chi'(\sigma+it)|^2|f'_{n,\chi^{c_0}}\circ \varphi_\chi(\sigma+it)|^2 \A^3 \big( \varphi_\chi(\sigma+it)\big) d\sigma dt dm(\chi)\\
 & \lesssim_{c_0,C} \veps^{2}
\end{align*}
by the same argument as that of Theorem \ref{prop:wcopositive}. On the other hand, arguing as above,
\begin{align*}
  &\int_{\ttinf}\!\int\!\!\int_{[0,1]^2\backslash \Omega_\chi} |u_\chi(\sigma+it)|^2 |\varphi_\chi'(\sigma+it)|^2|f'_{n,\chi^{c_0}}\circ \varphi_\chi(\sigma+it)|^2 \sigma d\sigma dt dm(\chi) \\
 & \lesssim_{c_0,C} \int_{\ttinf}\!\int\!\!\int_{[0,1]^2\backslash\Omega_\chi} |\varphi'_\chi(\sigma+it)|^2 | f_{n,\chi^{c_0}}'\circ \varphi_\chi(\sigma+it)|^2 \A^3 ( \varphi_\chi(\sigma+it))
 d\sigma dt dm(\chi).
\end{align*}
For a fixed $\chi\in\ttinf$, we do the nonunivalent change of variables $w=\varphi_\chi(s)$. Observe that we now have,
for all $\chi\in\ttinf$, $\varphi_\chi([0,1]^2\backslash \Omega_\chi)\subseteq [\theta,c_0+C]\times [-c_0-C,c_0+C]$. We get, using again the Nevanlinna counting function $\mathcal N_{\varphi_\chi,3}$ as
in the proof of Theorem \ref{prop:wcopositive},
\begin{align*}
  &\int_{\ttinf}\!\int\!\!\int_{[0,1]^2\backslash \Omega_\chi} |u_\chi(\sigma+it)|^2 |\varphi_\chi'(\sigma+it)|^2|f'_{n,\chi^{c_0}}\circ \varphi_\chi(\sigma+it)|^2 \sigma d\sigma dt dm(\chi) \\
&\quad\quad\lesssim_{c_0,C} \int_{\ttinf}\!\int_{-c_0-C}^{c_0+C}\!\int_{\theta}^{\sigma_0} |f_{n,\chi^{c_0}}'(\sigma+it)|^2\sigma^3 d\sigma dt dm(\chi)\\
&\quad\quad\lesssim_{c_0,C}\sum_{k=2}^{+\infty}(\log k)|a_{k,n}|^2 k^{-2\theta}
\end{align*}
and again it is easy to prove that this last quantity goes to zero as $n$ goes to $+\infty$ by weak convergence of $(f_n)$ to zero.
This concludes the proof that $uC_{\varphi}$ is a compact operator from $\mathcal A$ to $\mathcal H$.
\end{proof}

\subsubsection{Applications to connected components and compact differences}


Lemmas \ref{lem:wco} and  \ref{lem:normH},  Theorems \ref{prop:wcopositive} and \ref{prop:wcocompactnesspositive} lead us to the following sufficient conditions for two composition operators to be in the same component or for their difference to be compact.

\begin{theorem}\label{thm:samecomponent}
Let $\varphi_0$ and $\varphi_1\in\mathcal G$ with $\cha(\varphi_0)=\cha(\varphi_1)=:c_0\geq 1$ and
write them $\varphi_0=c_0s+\psi_0$, $\varphi_1=c_0s+\psi_1$. Assume moreover that there exists $C>0$ such that
\begin{itemize}
\item $|\varphi_0-\varphi_1|\leq C\min(\A \varphi_0,\A \varphi_1)$;
\item $|\psi_0|$, $|\psi_1|\leq C$;
\item $|\varphi_0'-\varphi_1'|\leq C$.
\end{itemize}
Then $\cphiz$ and $\cphiu$ belong to the same component of $\mathcal C(\mathcal H)$. If moreover we assume that
\[ |\varphi_0-\varphi_1|=o\big(\min(\A\varphi_0,\A \varphi_1)\big)\textrm{ and }\varphi_0'-\varphi_1'\to 0\textrm{ as }\min(\A\varphi_0,\A\varphi_1)\to 0 \]
then $\cphiz-\cphiu$ is compact.
\end{theorem}

We can deduce from this a result about non isolation of composition operators.

\begin{corollary}\label{cor:samecomponent}
Let $\varphi=c_0s+\psi\in\mathcal G$ with positive characteristic. Assume that $\psi$ and $\psi'$ are bounded on $\CC_+$ and that there exist $C>0$, $k\geq 1$, $t_0\in\RR$ such that, for all $s\in\CC_+$,
$$|\B (\psi(s))-t_0|\leq C \big( \A(\psi(s))\big)^{1/k}.$$
Then $C_\varphi$ is not isolated in $\mathcal C(\mathcal H)$.
\end{corollary}
\begin{proof}
Without loss of generality, we may assume that $t_0=0$.
Let us denote $\varphi_0=\varphi$ and let, for $\delta>0$, $\varphi_1=c_0s+\psi+\delta\psi^k$.
We claim that, provided $\delta$ is small enough, $\varphi_0$ and $\varphi_1$ satisfy the assumptions of Theorem \ref{thm:samecomponent}.
For $s\in\CC_+$, let us write
$$\A \varphi_1(s)=c_0 \A(s)+\A \psi(s)+\delta \A(\psi^k(s)).$$
Now,
$$ \A(\psi^k(s))=\sum_{j=0}^k a_j \big( \A\psi(s)\big)^j \big( \B\psi(s) \big)^{k-j}$$
for some coefficients $a_j$. Using our assumption on $\B \psi$ and the boundedness of $\A \psi$, there exist  $C_0>0$ such that, for all $s\in\mathbb C_+$,
$$ \A(\psi^k(s)) \geq -C_0\A(\psi(s)).$$
We set $\delta=1/2C_0$. We get for all $s\in\CC_+$
$$\A \varphi_1(s)\geq c_0 \A(s)+ \frac 12\A \psi(s)\geq \frac 12 \A \varphi_0(s).$$
Therefore $\varphi_1\in\mathcal G$. Furthermore, for all $s\in\CC_+$, with involved constants independent of $s$,
\begin{align*}
|\varphi_0(s)-\varphi_1(s)|^2&=\delta^2 |\psi(s)|^{2k}\\
&\lesssim \big( \A\psi(s)\big)^{2k}+\big( \B\psi(s)\big)^{2k}\\
&\lesssim \big( \A\psi(s)\big)^2\\
&\lesssim \big(\min(\A \varphi_0(s),\A \varphi_1(s))\big)^2,
\end{align*}
where, again, we used our assumption on $\B \psi$ and the boundedness of $\A \psi$.
\end{proof}

\begin{corollary}
Let $\varphi_0(s)=c_0s+c_1+\sum_{j=1}^d c_{q_j}q_j^{-s}$ be a linear symbol with $\cha(\varphi_0)\geq 1$. Then $\cphiz$ is not isolated
in $\mathcal C(\mathcal H)$. Moreover, if $\varphi_0$ has unrestricted range, there exists a continuous map
$\lambda\in[0,\delta]\mapsto \varphi_\lambda\in\mathcal G$ such that, for $\lambda\neq \lambda'$, $C_{\varphi_\lambda}-C_{\varphi_{\lambda'}}$ is never compact.
\end{corollary}
\begin{proof}
We only have to prove the statement when $\cphiz$ is not compact, namely when $\varphi_0$ has unrestricted range. In that case, we may assume that $c_{q_j}=-|c_{q_j}|$ and $c_1=\sum_{j=1}^d |c_{q_j}|$.
Since $s\mapsto q_j^{-s}$ maps $\CC_+$ into $\DD$, $\psi_0$ maps $\CC_+$ into some disk in $\CC_+$ which is internally tangent to the imaginary axis at $0$.
Therefore the assumptions of Corollary \ref{cor:samecomponent} are satisfied with $k=2$. Moreover, a look at the
proof of Theorem \ref{thm:samecomponent} shows that we can choose a continuous path $\lambda\mapsto\varphi_\lambda$
given by
$$\varphi_\lambda(s)=c_0s+c_1+\sum_{j=1}^d c_{q_j}q_j^{-s}+\lambda\left(c_1+\sum_{j=1}^d c_{q_j}q_j^{-s}\right)^2.$$
By Theorem \ref{thm:lcpolynomial}, for $\lambda\neq \lambda'$, $C_{\varphi_\lambda}-C_{\varphi_{\lambda'}}$ is never compact.
\end{proof}

In particular, the Shapiro-Sundberg conjecture keeps false in $\mathcal C(\mathcal H)$.
We now show the optimality of Theorem \ref{thm:lcpolynomial}, by proving that we cannot extend it to polynomial symbols of degree 3.

\begin{corollary}
Let $\varphi_0=c_0s+\psi_0\in\mathcal G$ with positive characteristic. Assume that $\psi_0$ and $\psi'_0$ are bounded on $\CC_+$ and that there exist $C>0$, $k\geq 1$, $t_0\in\RR$ such that, for all $s\in\CC_+$,
$$|\B (\psi_0(s))-t_0|\leq C \big( \A(\psi_0(s))\big)^{1/k}.$$
Let $\delta>0$ and $\varphi_1=c_0s +\psi_0+\delta\psi_0^{k+1}$. Then, for $\delta>0$ sufficiently small, $\varphi_1\in\mathcal G$ and $C_{\varphi_0}-C_{\varphi_1}$ is compact.
\end{corollary}
\begin{proof}
The proof follows that of Corollary \ref{cor:samecomponent}. With the same computations, still assuming $t_0=0$, one now gets
\[ |\varphi_0(s)-\varphi_1(s)|^2\lesssim \big(\min(\A\varphi_0(s),\A\varphi_1(s))\big)^{2(k+1)/k} \]
which is enough to ensure that $|\varphi_0-\varphi_1|=o\big(\min (\A\varphi_0,\A\varphi_1)\big)$ as
$\min (\A\varphi_0,\A\varphi_1)$ tends to zero. We conclude by applying Theorem \ref{thm:samecomponent} since $\varphi_1'-\varphi_0'=\delta (k+1)\psi_0'\psi_0^k$ and since the assumptions imply that $\psi_0$ tends to zero as $\min(\A\varphi_0,\A\varphi_1)$ tend to zero.
\end{proof}

\begin{example}
Let $\varphi_0(s)=s+(1-2^{-s})$ and $\varphi_1(s)=s+(1-2^{-s})+\delta (1-2^{-s})^3$ for some sufficiently small $\delta>0$. Then $C_{\varphi_1}-C_{\varphi_0}$ is compact whereas neither $C_{\varphi_0}$ nor $C_{\varphi_1}$ is compact.
In particular, Theorem \ref{thm:lcpolynomial} is optimal.
\end{example}

\subsection{The case of zero characteristic}

\subsubsection{Weighted composition operators from Bergman spaces to Hardy spaces}

In this section, we give a sufficient condition for a weighted composition operator $u C_\varphi$ to
induce a bounded, respectively compact, operator from the Bergman space $A^2_1(\mathbb D)$ to the Hardy space $H^2(\mathbb D)$. Here $A^2_1(\mathbb D)$ is the set of analytic functions $f$ on the disk such that
\[ \|f\|_{A_1^2(\mathbb D)}^2:=\int_{\mathbb D} |f(z)|^2 (1-|z|^2)dz<+\infty.\]
The continuity of $uC_\varphi:A^2_1(\mathbb D)\to H^2(\mathbb D)$ has been characterized in terms of Carleson measures
in \cite{MT} or using integral operators in \cite{Su}. Inspired by \cite{rn}, we shall deduce from the results of \cite{MT} a sufficient condition which will be easy to testify. Let us introduce some terminology.
 For a point $\xi$ in the boundary of the unit disk and $\delta>0$, we define the Carleson window $S(\xi,\delta)=\{z\in\mathbb D : |z-\xi|<\delta\}$.

Let $\varphi$ be an analytic self-map of $\mathbb D$ and let $u$ be a bounded analytic function on $\mathbb D$. We define a positive Borel measure $|u|^2 m_1\varphi^{-1}$ on the closed unit disk $\overline {\mathbb D}$ by assigning to each Borel set $E$ the value:
\[ |u|^2 m_1\varphi^{-1}(E)=\int_{{\varphi}^{-1}(E)}|u(\xi)|^2 dm_1(\xi), \]
where we still denote by $\varphi$ and $u$ the radial limit functions of $\varphi$ and $u$ and where $m_1$ is the normalized Lebesgue measure on $\TT$.

The following lemma is \cite[Proposition 2]{MT}.

\begin{lemma}\label{lem:mt}
Let $\varphi$ be an analytic self-map of $\mathbb D$ and let $u$ be a bounded analytic function on $\mathbb D$ not identically equal to zero. Then $uC_\varphi:A^2_1\to H^2$ is bounded if and only if $\varphi$ has radial limits of modulus strictly less than $1$ almost everywhere and
\[ \sup_{S(\xi,\delta)} \frac{|u|^2 m_1\varphi^{-1}(S(\xi,\delta))}{\delta^3}<+\infty. \]
Moreover
\[  \|u C_{\varphi}\|_{A^2_1\to H^2} \approx \left( \sup_{S(\xi,\delta)} \frac{|u|^2 m_1\varphi^{-1}(S(\xi,\delta))}{\delta^3} \right)^{1/2}. \]
If $\sup_{S(\xi,\delta)}|u|^2 m_1\varphi^{-1}(S(\xi,\delta))/\delta^3$ goes to zero uniformly in $\xi$ as $\delta$ tends to zero, then $uC_\phi$ is a compact operator from $A_1^2$ to $H^2$.
\end{lemma}

\begin{remark}
In the statement of \cite{MT}, the exponent $1/2$ in the last display does not appear. Nevertheless a look at the proof as well as homogeneity considerations show that it has to appear.
\end{remark}

\begin{corollary}\label{thm:weighted}
Let $\varphi$ be an analytic self-map of $\mathbb D$ which has radial limits of modulus strictly less than $1$ almost everywhere and let $u$ be a bounded analytic function on $\mathbb D$ not identically equal to zero.  Then
$$\|u C_{\varphi}\|_{A^2_1\to H^2}\lesssim\left(\frac{1+|\varphi(0)|}{1-|\varphi(0)|}\right)^{1/2}
\sup_{\xi\in\mathbb T}  \frac{|u(\xi)|}{1-|\varphi(\xi)|}.$$
If moreover
\[ \sup_{|\varphi(\xi)|\geq 1-\delta}  \frac{|u(\xi)|}{1-|\varphi(\xi)|}\xrightarrow{\delta\to 0}0\]
then $uC_\varphi$ is a compact operator from $A_1^2$ to $H^2$.
\end{corollary}
\begin{proof}
We simply write, for $\xi\in\mathbb T$ and $\delta>0$,
\begin{align*}
 \frac{|u|^2 m_1\varphi^{-1}(S(\xi,\delta))}{\delta^3}&= \frac 1{\delta^3} \int_{{\varphi}^{-1}(S(\xi,\delta))} |u|^2 dm_1\\
&\leq \left(\sup_{\xi_0\in\mathbb T}  \frac{|u(\xi_0)|}{1-|\varphi(\xi_0)|}\right)^2 \times  \frac 1{\delta^3} \int_{{\varphi}^{-1}(S(\xi,\delta))} (1-|\varphi|)^2 dm_1\\
&\leq \left(\sup_{\xi_0\in\mathbb T}  \frac{|u(\xi_0)|}{1-|\varphi(\xi_0)|}\right)^2 \times  \frac 1{\delta} m_1\big({\varphi}^{-1}(S(\xi,\delta))\big).
\end{align*}
It follows from the Carleson measure characterisation of the continuity of composition operators
on $H^2(\mathbb D)$ that
\(  \frac 1{\delta} m_1\big({\varphi}^{-1}(S(\xi,\delta))\big) \lesssim \|C_{\varphi}\|_{H^2\to H^2}^2
\leq \frac{1+|\varphi(0)|}{1-|\varphi(0)|}. \)
\end{proof}

We translate this statement to get a result on the invariant Hardy space $H^2_i(\CC_+)$.

\begin{lemma}\label{lem:weightedhalfplane}
Let $\phi$ be an analytic self-map of $\mathbb C_+$ which has radial limits of positive real part almost everywhere and let $U$ be a bounded analytic function on $\mathbb C_+$ not identically equal to zero.  Then
for all $F\in H^2_i(\mathbb C_+)$,
\[ \| U\cdot F'\circ \phi\|_{H^2_i(\CC_+)} \lesssim \left( \frac{ |1+\phi(1)|^2}{\A \phi(1) } \right)^{1/2} \sup_{t\in\mathbb R}
\frac{|U(it)|\times |1+\phi(it)|^2}{\A \phi(it)}\|F\|_{H^2_i(\CC_+)}. \]
If, moreover,
\[ \sup\left\{\frac{|U(it)|\times |1+\phi(it)|^2}{\A \phi(it)}:\ \frac{\A\phi(it)}{|1+\phi(it)|^2}\leq\delta\right\}\xrightarrow{\delta\to 0}0\]
then the operator $F\mapsto U\cdot F'\circ\phi$ is a compact operator of $H^2_i(\CC_+)$.
\end{lemma}
\begin{proof}
Let $\omega$ be the Cayley map, $\omega(z)=\frac{1+z}{1-z}$ and let us set $f=F\circ \omega$, $u=U\circ\omega$,
$\varphi=\omega^{-1}\circ\phi\circ \omega$. Observe that
\(U F'\circ\phi= g\circ \omega^{-1} \) where
\[ g= u \times (\omega^{-1})'\circ \omega \times f'\circ\varphi = \frac{(z-1)^2}2 u \times f'\circ \varphi. \]
A small computation shows that
\[ \sup_{\xi\in\mathbb T}  \frac{|u(\xi)|}{1-|\varphi(\xi)|^2} \approx  \sup_{t\in\mathbb R}
\frac{|U(it)|\times |1+\phi(it)|^2}{\A \phi(it)} \]
and
\[  \frac{1+|\varphi(0)|}{1-|\varphi(0)|} \approx \frac{ |1+\phi(1)|^2}{\A \phi(1) }. \]
Therefore
\begin{align*}
 \| U\cdot F'\circ \phi\|_{H^2_i(\CC_+)}& = \| g\|_{H^2(\DD)}\\
&\lesssim  \left( \frac{ |1+\phi(1)|^2}{\A \phi(1) } \right)^{1/2} \sup_{t\in\mathbb R}
\frac{|U(it)|\times |1+\phi(it)|^2}{\A \phi(it)} \| f'\|_{A^2_1(\DD)}\\
&\lesssim  \left( \frac{ |1+\phi(1)|^2}{\A \phi(1) } \right)^{1/2} \sup_{t\in\mathbb R}
\frac{|U(it)|\times |1+\phi(it)|^2}{\A \phi(it)} \| f\|_{H^2(\DD)}\\
&\lesssim  \left( \frac{ |1+\phi(1)|^2}{\A \phi(1) } \right)^{1/2} \sup_{t\in\mathbb R}
\frac{|U(it)|\times |1+\phi(it)|^2}{\A \phi(it)} \|F\|_{H^2_i(\CC_+)}.
\end{align*}
The statement about compactness follows from the second half of Corollary \ref{thm:weighted}.
\end{proof}

\begin{remark}
We could improve the statement of Lemma \ref{lem:weightedhalfplane} by replacing everywhere
the term $|U(it)|\times |1+\phi(it)|^2/\A \phi(it)$ by $|U(it)|\times |1+\phi(it)|^2/(\A \phi(it)\times (1+t^2))$. Indeed, we apply Corollary \ref{thm:weighted} not to $u$ but to $(z-1)^2u/2$ which allows
this better inequality. Nevertheless, since we plan to apply this result to vertical limits, with a uniform
bound, the term $1+t^2$ would not be helpful to us.
\end{remark}

\subsubsection{Applications to connected components and compact differences}

\begin{theorem}\label{thm:samecomponentzero}
Let $\varphi_0$, $\varphi_1\in\mathcal G$ with $\cha(\varphi_0)=\cha(\varphi_1)=0$. Assume that there exists $C>0$ such that
 $|\varphi_0-\varphi_1|\leq C\min\left(\frac{\A \varphi_0-1/2}{|1+\varphi_0|^2},\frac{\A \varphi_1-1/2}{|1+\varphi_1|^2}\right)$.
Then $\cphiz$ and $\cphiu$ belong to the same component of $\mathcal C(\mathcal H)$. If, \ moreover, \
$|\varphi_0-\varphi_1|= o\left(\min\left(\frac{\A \varphi_0-1/2}{|1+\varphi_0|^2},\frac{\A \varphi_1-1/2}{|1+\varphi_1|^2}\right)\right)$ as $\min\left(\frac{\A \varphi_0-1/2}{|1+\varphi_0|^2},\frac{\A \varphi_1-1/2}{|1+\varphi_1|^2}\right)$ tends to zero, then $\cphiz-\cphiu$ is compact.
\end{theorem}
\begin{proof}
We start as in the proof of Theorem \ref{thm:samecomponent} but now we choose the measure $d\mu=\frac{ dt}{1+t^2}$.
Taking into account that the symbols have zero characteristic, we know that for all $\lambda,\ \lambda'\in[0,1]$,
\[
\begin{array}{rcl}
\left\|C_{\varphi_\lambda}(f)-C_{\varphi_{\lambda'}}(f)\right\|^2 &\leq&\displaystyle |\lambda'-\lambda|^2
\int_\lambda^{\lambda'}\!\int_{\ttinf}\!\int_{\mathbb R}\left|(\varphi_1)_\chi(it)-(\varphi_0)_\chi(it)\right|^2\\[4mm]
&&\quad\quad\quad\quad\quad\quad\quad\times
\displaystyle |f'((\varphi_r)_\chi(it))|^2\frac{dt}{1+t^2} dm(\chi)dr.
\end{array}
\]
Now $\mathcal H$ embeds into $H^2_i(\CC_+)$ via the following inequality (\cite{br}): for all $g\in\mathcal H$,
\[ \|g\|_{H^2_i(\mathbb C_{1/2})}\leq \sqrt 2\|g\|_{\mathcal H}. \]
Therefore we just need to prove that there exists $A>0$ such that, for all $\chi\in\ttinf$, for all $r\in[0,1]$ and for all $g\in H^2_i(\CC_{1/2})$,
\[ {\|(\varphi_1-\varphi_0)_\chi \cdot g'\circ (\varphi_r)_\chi\|_{H^2_i(\CC_+)} } \leq A \|g\|_{H^2_i(\CC_{1/2})}. \]
For a fixed $a>0$, the condition $\frac{\A z-1/2}{|1+z|^2}\geq a$ is equivalent to say that $z$ belongs to some disk.
Since every disk is convex, we get that for all $r\in[0,1]$,
\[ |\varphi_0-\varphi_1| \leq C\frac{\A \varphi_r-1/2}{|1+\varphi_r|^2}. \]
By vertical translation, we also have
\[ |(\varphi_0)_\chi-(\varphi_1)_\chi| \leq C\frac{\A (\varphi_r)_\chi-1/2}{|1+(\varphi_r)_\chi|^2}. \]
Moreover, since $\varphi_0$ and $\varphi_1$ converge uniformly on $\CC_{1/2}$,
there exists some compact subset $K\subseteq \CC_{1/2}$ such that, for all $\chi\in\ttinf$ and all $r>0$, $(\varphi_r)_\chi(1)$ belongs to $K$. Therefore, Lemma \ref{lem:weightedhalfplane} gives the desired result.
The proof of the compactness of $\cphiz-\cphiu$ follows the same lines and is left to the reader.
\end{proof}

We then get a result similar to Corollary \ref{cor:samecomponent} for composition operators with zero characteristic.

\begin{corollary}\label{cor:samecomponentzero}
Let $\varphi\in\mathcal G$ with zero characteristic and such that $|\varphi|$ is bounded. Assume that there exist $C>0$, $k\geq 1$, $t_0\in\RR$ such that, for all $s\in\CC_+$,
$$|\B (\varphi(s))-t_0|\leq C \left( \A(\varphi(s))-\frac 12\right)^{1/k}.$$
Then $C_\varphi$ is not isolated in $\mathcal C(\mathcal H)$.
\end{corollary}
\begin{proof}
We set $\varphi_0=\varphi$ and let, for $\delta>0$, $\varphi_1=\varphi+\delta\left(\varphi-\frac12\right)^k$.
As in the proof of Corollary \ref{cor:samecomponent}, $\varphi_1\in\mathcal G$ provided $\delta$ is small enough. Moreover, the assumptions of Corollary \ref{cor:samecomponentzero} are satisfied since
\begin{align*}
|\varphi_0(s)-\varphi_1(s)|^2&=\delta^2 \left|\varphi(s)-\frac 12\right|^{2k}\\
&\lesssim \left( \A\varphi(s)-\frac 12\right)^{2k}+\big( \B\varphi(s)\big)^{2k}\\
&\lesssim\left( \A\varphi(s)-\frac 12\right)^{2}\\
&\lesssim \frac{\left( \A\varphi(s)-\frac 12\right)^{2}}{|\varphi+1|^4}
\end{align*}
since $|\varphi|$ is bounded.
\end{proof}

We deduce from this result that the Shapiro-Sundberg conjecture remains false for composition operators with zero characteristic.
\begin{example}
Let $\varphi_0(s)=\frac 12+(1-2^{-s})$ and $\varphi_1(s)=\varphi_0(s)+\delta(1-2^{-s})^2$. Then, for $\delta>0$ sufficiently small, $\varphi_0,\ \varphi_1\in\mathcal G$, $\cphiz$ and $\cphiu$ belong to the same component of $\mathcal C(\mathcal H)$ but $\cphiz-\cphiu$ is not compact.
\end{example}
\begin{proof}
The first half of the statement comes from Corollary \ref{cor:samecomponentzero} (and its proof). The second half is a consequence of Theorem \ref{thm:lcpolynomial} since neither $\varphi_0$ nor $\varphi_1$ is compact by \cite[Theorem 2]{bb}.
\end{proof}

\subsection{Order of contact, connected components and compact differences}

In this final section, we show how the conditions of Theorem \ref{thm:samecomponent} can be satisfied from conditions on the derivatives of $\mathcal B\psi$.
We need a notion of order of contact, which is maybe less natural than in one variable (see \cite{km}).

Let $\varphi\in\mathcal G$ be a Dirichlet polynomial symbol with $\cha(\varphi)\geq 1$. Let $d$ be the smallest integer such that
\[ \varphi(s)=c_0 s+\sum_{p^+(n)\leq p_d}c_n n^{-s}. \]
Then $\bpsi$ maps $\DD^d$ into $\CC_+$. Denote by $J_\varphi:\RR^d\to [0,+\infty[$, $J_\varphi(\theta_1,\dots,\theta_d)=\A \bpsi (e^{i\theta_1},\cdots,e^{i\theta_d})$. Recall also
that $\Gamma(\bpsi)$ is defined by $\Gamma(\bpsi)=\{z\in\mathbb T^d: \A\bpsi(z)=0\}$ (this set is nonempty provided $\varphi$ has unrestricted range). Let $z= (e^{i\theta_1^0},\dots,e^{i\theta_d^0})\in \Gamma(\bpsi)$. Then $J_\varphi$ attains its minimum at $(\theta_1^0,\dots,\theta_d^0)$.
\begin{definition}
We say that $\varphi$ has a  {\bf Dirichlet contact of order $2n$ at $z=(e^{i\theta_1^0},\ldots,e^{i\theta_d^0})$} provided there exists $c>0$ and a neighbourhood $\mathcal U$ of  $(\theta_1^0,\dots,\theta_d^0)$ in $\RR^d$ such that, for all $(\theta_1,\dots,\theta_d)\in \mathcal U$,
$$J_\varphi(\theta_1,\dots,\theta_d)\geq c\big( (\theta_1-\theta_1^0)^2+\cdots+(\theta_d-\theta_d^0)^2\big)^n.$$
\end{definition}

Let us explain this terminology. Without loss of generality, we may assume that $\theta_1^0=\cdots=\theta_d^0=0$ and $\bpsi(1,\ldots,1)=0$.
The Taylor expansion of $\bpsi$ at $z$ may be written
$$\bpsi(e^{i\theta_1},\dots,e^{i\theta_d})=\sum_{j=1}^d a_j (1-e^{i\theta_j})+O\left(\theta_1^2+\cdots+\theta_d^2\right)$$
with each $a_j\geq 0$ (see \cite{bb}).
In particular, there exists a neighbourhood  $\mathcal V$ of $(0,\dots,0)$ in $\mathbb R^d$ such that, for all $(\theta_1,\dots,\theta_d)\in \mathcal V$,
$$|\B \bpsi(e^{i\theta_1},\dots,e^{i\theta_d})|\leq C (\theta_1^2+\cdots+ \theta_d^2)^{1/2}.$$
Therefore, for $(\theta_1,\dots,\theta_d)\in \mathcal U\cap\mathcal V$,
$$ |\B \bpsi(e^{i\theta_1},\dots,e^{i\theta_d})|^{2n}\lesssim  \A \bpsi(e^{i\theta_1},\dots,e^{i\theta_d}).$$
Hence, if $\varphi$ has a Dirichlet contact of order $2n$ at $z$, the Bohr lift $\mathcal B\psi $ has a contact of order $2n$ at $z$.

\smallskip

For $w\in\mathbb T^d$ and $\sigma>0$, we denote by $w_\sigma=(2^{-\sigma}w_1,\dots, p_d^{-\sigma}w_d)$.

\begin{lemma}
Let $\varphi_0,\varphi_1\in\mathcal G$ be Dirichlet polynomial symbols with $\cha(\varphi_0)=\cha(\varphi_1)\geq 1$. Let $z\in\Gamma(\mathcal B{\psi_0})\cap \Gamma(\mathcal B{\psi_1})$ such that
\begin{itemize}
\item $\mathcal B{\psi_0}(z)=\mathcal B{\psi_1}(z)$;
\item $\varphi_0$ and $\varphi_1$ have a Dirichlet contact of order $2n$ at $z$;
\item for $|\alpha|\leq 2n-1$, $\partial_\alpha \mathcal B{\psi_0}(z)=\partial_\alpha \mathcal B{\psi_1}(z)$.
\end{itemize}
 Then there exist a neighbourhood $\mathcal U$ of $z$ in $\mathbb T^d$, $\sigma_0>0$ and $C>0$ such that, for all $w\in\mathcal U$, for all $\sigma\in [0,\sigma_0]$,
$$|\mathcal B{\psi_0}(w_\sigma)-\mathcal B{\psi_1}(w_\sigma)|\leq C \min \big(\A \mathcal B{\psi_0}(w_\sigma),\A \mathcal B{\psi_1}(w_\sigma)\big).$$
Moreover, if we assume that $\partial_\alpha \mathcal B\psi_0(z)=\partial_\alpha\mathcal B\psi_1(z)$ for $|\alpha|\leq 2n$, then
\[ |\mathcal B{\psi_0}(w_\sigma)-\mathcal B{\psi_1}(w_\sigma)| = o\left( \min \big(\A \mathcal B{\psi_0}(w_\sigma),\A \mathcal B{\psi_1}(w_\sigma)\big)\right) \]
as $w_\sigma\to z$.
\end{lemma}
\begin{proof}
We may assume that $z=\mathbf e$ and that $\mathcal B\psi_0(\mathbf e)=\mathcal B\psi_1(\mathbf e)=0$. We set $m=2n-1$ or $m=2n$, following the assumptions that we have made. These assumptions imply that, looking at some $w=(e^{i\theta_1},\dots,e^{i\theta_d})$,
\begin{align*}
|\mathcal B\psi_0(w_\sigma)-\mathcal B\psi_1(w_\sigma)|&=|\mathcal B\psi_0(w_\sigma)-1+1-\mathcal B\psi_1(w_\sigma)|\\
&=\Bigg|\sum_{|\alpha|=m+1}a_\alpha(1-2^{-\sigma}e^{i\theta_1})^{\alpha_1}\cdots
(1-p_d^{-\sigma}e^{i\theta_d})^{\alpha_d}\\
&\quad\quad+o(\|w_\sigma-\mathbf e\|^{m+1})\Bigg|\\
&\leq C_1\left(\sigma^{m+1}+|\theta_1|^{m+1}+\cdots +|\theta_d|^{m+1}\right)
\end{align*}
for some constant $C_1>0$.

On the other hand, let $i=0,1$ and let us write the Taylor  expansion of $\mathcal B\psi_i$ at $w$ as
$$\mathcal B\psi_i(w_\sigma)=\mathcal B\psi_i(w)+\sum_{j=1}^d a_{i,j}(w) (1-p_j^{-\sigma})+F_i(w,\sigma)$$
where the functions $a_{i,j}$ are continuous and $|F_i(w,\sigma)|\leq C_2\sigma^2$. Expanding $1-p_j^{-\sigma}$,
this can been rewritten as
$$\mathcal B\psi_i(w_\sigma)=\mathcal B\psi_i(w)+\left(\sum_{j=1}^d a_{i,j}(w)\log p_j \right)\sigma +G_i(w,\sigma)$$
with $|G_i(w,\sigma)|\leq C_3\sigma^2$.  Now, we know (see \cite{bb}) that $a_{i,j}(\mathbf e)\geq 0$ and $a_{i,j_0}(\mathbf e)>0$ for at least one $j_0\in\{1,\dots,d\}$.
From the continuity of the functions $a_{i,j}$, we deduce the existence of a neighbourhood $\mathcal U$ of $\mathbf e$ in $\mathbb T^d$, of $C_4>0$  and of $\sigma_0>0$ such that, for all $w\in\mathcal U$ and for all $\sigma\in [0,\sigma_0]$,
\begin{align*}
\A  \mathcal B\psi_i(w_\sigma)&\geq \A \mathcal B\psi_i(w)+C_4\sigma\\
&\geq C_4\sigma+C_5(\theta_1^{2n}+\cdots+\theta_d^{2n})
\end{align*}
where the last inequality holds because $\varphi_i$ has a Dirichlet contact of order $2n$ at $\mathbf e$. This yields the following lemma.
\end{proof}

\begin{lemma}
Let $\varphi_0,\varphi_1\in\mathcal G$ be Dirichlet polynomial symbols with $\cha(\varphi_0)=\cha(\varphi_1)\geq 1$. Assume that $\Gamma(\mathcal B{\psi_0})=\Gamma( \mathcal B{\psi_1})$ and that, for all $z\in\Gamma(\mathcal B{\psi_0})$, there exists $n\in\mathbb N$ such that
\begin{itemize}
\item $\mathcal B{\psi_0}(z)=\mathcal B{\psi_1}(z)$;
\item $\varphi_0$ and $\varphi_1$ have a Dirichlet contact of order $2n$ at $z$;
\item for $|\alpha|\leq 2n-1$, $\partial_\alpha \mathcal B{\psi_0}(z)=\partial_\alpha \mathcal B{\psi_1}(z)$.
\end{itemize}
 Then there exists $C>0$ such that, for all $w\in\mathbb T^d$ and all $\sigma\geq 0$,
$$|\mathcal B{\psi_0}(w_\sigma)-\mathcal B{\psi_1}(w_\sigma)|\leq C \min \big( \A \mathcal B{\psi_0}(w_\sigma),\A\mathcal B{\psi_1}(w_\sigma)\big).$$
Moreover, if we assume that $\partial_\alpha \mathcal B\psi_0(z)=\partial_\alpha\mathcal B\psi_1(z)$ for $|\alpha|\leq 2n$, then
\[ |\mathcal B{\psi_0}(w_\sigma)-\mathcal B{\psi_1}(w_\sigma)| = o\left( \min \big(\A \mathcal B{\psi_0}(w_\sigma),\A \mathcal B{\psi_1}(w_\sigma)\big)\right) \]
as $\min \big(\A \mathcal B{\psi_0}(w_\sigma),\A \mathcal B{\psi_1}(w_\sigma)\big)$ goes to zero.
\end{lemma}
\begin{proof}
This follows from the previous lemma and a compactness argument.
\end{proof}

\begin{corollary}
Let $\varphi_0,\varphi_1\in\mathcal G$ be Dirichlet polynomial symbols with $\cha(\varphi_0)=\cha(\varphi_1)\geq 1$. Assume that $\Gamma(\mathcal B{\psi_0})=\Gamma( \mathcal B{\psi_1})$ and that, for all $z\in\Gamma(\mathcal B{\psi_0})$, there exists $n\in\mathbb N$ such that
\begin{itemize}
\item $\mathcal B{\psi_0}(z)=\mathcal B{\psi_1}(z)$;
\item $\varphi_0$ and $\varphi_1$ have a Dirichlet contact of order $2n$ at $z$;
\item for $|\alpha|\leq 2n-1$, $\partial_\alpha \mathcal B{\psi_0}(z)=\partial_\alpha \mathcal B{\psi_1}(z)$.
\end{itemize}
 Then $\cphiz$ and $\cphiu$ belong to the same component of $\mathcal C(\mathcal H)$.
 Moreover, if we assume that $\partial_\alpha \mathcal B\psi_0(z)=\partial_\alpha\mathcal B\psi_1(z)$ for $|\alpha|\leq 2n$, then $\cphiz-\cphiu$ is compact.
\end{corollary}
\begin{proof}
The previous lemma
 implies that the assumptions of Theorem \ref{thm:samecomponent} are satisfied.
\end{proof}

Combining Theorem \ref{thm:essentialnormcomposition} and the previous corollary, we finally get:
\begin{theorem}
Let $\varphi_0,\varphi_1\in\mathcal G$ be Dirichlet polynomial symbols with $\cha(\varphi_0)=\cha(\varphi_1)\geq 1$.
Assume that for all $z\in\Gamma(\mathcal B\psi_0)$ (resp. $z\in\Gamma(\mathcal B(\psi_1))$, $\varphi_0$ (resp. $\varphi_1$) has Dirichlet order of contact $2$ at $z$. Then the following assertions are equivalent:
\begin{enumerate}[(i)]
\item $\cphiz-\cphiu$ is compact.
\item $\Gamma(\mathcal B\psi_0)=\Gamma(\mathcal B\psi_1)$ and, for all $z\in \Gamma(\mathcal B\psi_0)$, $\mathcal B\psi_0(z)=\mathcal B\psi_1(z)$, $\partial_\alpha \mathcal B\psi_0(z)=\partial_\alpha\mathcal B\psi_1(z)$ for $|\alpha|\leq 2$.
\end{enumerate}
\end{theorem}

\subsection*{Acknowledgment.}
F. Bayart was supported by the grant ANR-17-CE40-0021 of the French National Research Agency ANR (project Front),
M. Wang and X. Yao were supported by National Science Foundation of China (Nos. 12171373, 11771340, 11701434).

\end{document}